\documentclass[12pt, reqno]{amsart}

\usepackage[margin=1.15in]{geometry}
\usepackage{amscd,amssymb, amsmath,amsfonts, wasysym, mathrsfs, enumitem, stmaryrd, mathtools,hhline,color,enumitem, tikz-cd}
\usepackage{graphicx}
\usepackage[all, cmtip]{xy}
\usepackage[T1]{fontenc}
\usepackage{newtxtext}
\usepackage{newtxmath}
\usepackage{textcomp}
\usepackage{url}
\usepackage{array}

\usepackage{thmtools}
\usepackage{enumitem}
\usepackage[pagebackref=true,colorlinks=true, linkcolor=violet,  citecolor=hot, urlcolor=magenta]{hyperref}%make all the references and links clickable
\usepackage{nameref}
\usepackage{cleveref}

\setlist[itemize]{wide = 0pt, labelwidth = 2em, labelsep*=0em, itemindent = 0pt, leftmargin = \dimexpr\labelwidth + \labelsep\relax, noitemsep,topsep = 1ex,}
\setlist[enumerate]{wide = 0pt, labelwidth = 2em, labelsep*=0em, itemindent = 0pt, leftmargin = \dimexpr\labelwidth + \labelsep\relax, noitemsep,topsep = 1ex}

\definecolor{hot}{RGB}{65,105,225}

\newtheorem{theorem}{Theorem}[section]
\newtheorem{proposition}[theorem]{Proposition}

\newtheorem{claim}[theorem]{Claim}

\newtheorem{corollary}[theorem]{Corollary}
\newtheorem{conj}[theorem]{Conjecture}
\newtheorem{lemma}[theorem]{Lemma}
\theoremstyle{definition}

\newtheorem{definition}[theorem]{Definition}
\newtheorem{question}[theorem]{Question}
\newtheorem{remark}[theorem]{Remark}

\newtheorem{example}[theorem]{Example}
\newtheorem*{ex*}{Example}

\usepackage[hyperpageref]{backref}

\theoremstyle{plain}
\newlist{thmlist}{enumerate}{1}
\setlist[thmlist]{wide = 0pt, labelwidth = 2em, labelsep*=0em, itemindent = 0pt, leftmargin = \dimexpr\labelwidth + \labelsep\relax, noitemsep,topsep = 1ex, font=\normalfont, label=(\roman*), ref=\thetheorem.(\roman{thmlisti})}

\addtotheorempostheadhook[theorem]{\crefalias{thmlisti}{theorem}}

\addtotheorempostheadhook[corollary]{\crefalias{thmlisti}{corollary}}

\addtotheorempostheadhook[proposition]{\crefalias{thmlisti}{proposition}}

\addtotheorempostheadhook[definition]{\crefalias{thmlisti}{definition}}

\addtotheorempostheadhook[lemma]{\crefalias{thmlisti}{lemma}}

\addtotheorempostheadhook[remark]{\crefalias{thmlisti}{remark}}

\crefname{lemma}{Lemma}{Lemmas} 
\crefname{conjecture}{Conjecture}{Conjectures}
\crefname{theorem}{Theorem}{Theorems}
\crefname{proposition}{Proposition}{Propositions}
\crefname{definition}{Definition}{Definitions}
\crefname{remark}{Remark}{Remarks}
\crefname{corollary}{Corollary}{Corollaries}

\newcommand\sP{{\mathcal P}}
\newcommand\cP{{\mathcal P}}

\newcommand\sF{{\mathcal F}}

\newcommand\GL{\mathrm{GL}}

\def\R{\mathbb{R}}

\def\Z{\mathbb{Z}}

\def\bV{\mathbb{V}}

\DeclareMathOperator{\Supp}{Supp}                % Supp
\DeclareMathOperator{\reg}{reg}                  % reg
                  
\DeclareMathOperator{\norm}{norm}

\DeclareMathOperator{\univ}{univ}

\DeclareMathOperator{\im}{Im}

\def\C{\mathbb{C}}
\def\bR{\mathbb{R}}
\def\cM{{\mathscr{M}}}

\def\End{{\rm {End}}}

\def\bC{\mathbb{C}}
\def\bP{\mathbb{P}}

\def\cD{\mathscr{D}}
\def\cO{\mathcal{O}}

\def\Q{\mathbb{Q}}

\author[Y. Deng]{Ya Deng}
\address{CNRS, Institut \'Elie Cartan de Lorraine, Universit\'e de Lorraine, Site de
	Nancy,   54506 Vand\oe uvre-lès-Nancy, France} 
\email{ya.deng@math.cnrs.fr}
\urladdr{https://ydeng.perso.math.cnrs.fr} 

\author[B. Wang]{Botong Wang}
\address{Department of Mathematics, University of Wisconsin-Madison, 480 Lincoln Drive, Madison WI 53706-1388, USA}
\email {wang@math.wisc.edu}
\urladdr{https://people.math.wisc.edu/~bwang274/}

\title[Linear Chern-Hopf-Thurston conjecture]{Linear Chern-Hopf-Thurston conjecture}

\begin{document} 
	\begin{abstract}
		If \(X\) is a closed \(2n\)-dimensional aspherical manifold, i.e., the universal cover of \(X\) is contractible, then the Chern-Hopf-Thurston conjecture predicts that \((-1)^n\chi(X)\geq 0\). We prove this conjecture when \(X\) is a complex projective manifold whose fundamental group admits an almost faithful linear representation over any field. In fact, we prove a much stronger statement that if \(X\) is a complex projective manifold with large fundamental group and \(\pi_1(X)\) admits an almost faithful linear representation, then \(\chi(X, \mathcal{P})\geq 0\) for any perverse sheaf \(\mathcal{P}\) on \(X\). 
		
		To prove this, we introduce a vanishing cycle functor of multivalued one-forms and apply techniques from non-abelian Hodge theory, both in archimedean and non-archimedean settings. These techniques allow us to deduce the desired positivity from the geometric properties of pure and mixed period maps.
		%To prove this, we introduce a vanishing cycle functor of multivalued one-forms. Then using techniques from non-abelian Hodge theories in both archimedean and non-archimedean settings, we deduce the desired positivity from the geometry of pure and mixed period maps. 
	\end{abstract}
	\maketitle
	\tableofcontents
	
	\section{Introduction}
	%	\subsection{Conjectures and main theorems}
	In the 1930s, Hopf conjectured that a closed $2n$-dimensional compact Riemannian manifold
	$X$ with non-positive sectional curvature satisfies $(-1)^n\chi(X)\geq 0$ (see \cite[Problem 10]{Yau}). %Singer generalized Hopf's conjecture to aspherical manifolds and predicts the following. 
	This conjecture follows from the Gauss-Bonnet theorem when $n=2$ (see \cite{C}), but it remains widely open in higher dimensions. 
	By the Cartan-Hadamard theorem, a compact Riemannian manifold $X$ with non-positive sectional curvature is aspherical, i.e., the universal cover of $X$ is contractible. Thurston, therefore, proposed the following generalization of Hopf's conjecture (see \cite[Problem 4.10]{Kirby} and \cite{Davis}).
%	Singer proposed to study Hopf's conjecture using $L^2$-cohomology groups of the universal cover of $X$, which naturally lead to the following more general conjecture.  
	\begin{conj}[Chern-Hopf-Thurston]\label{conj:Chern-Hopf-Thurston}
		Let $X$ be a closed $2n$-dimensional manifold. If $X$ is aspherical,  then 
		\[
		(-1)^n\chi(X)\geq 0.
		\]
	\end{conj}
	In this general form, the conjecture is not even known in dimension 4. Charney and Davis also considered the Chern-Hopf-Thurston conjecture for piecewise Euclidean manifolds, which lead to the Charney-Davis conjecture about flag complexes (see \cite{CD}). Singer proposed to study the Chern-Hopf-Thurston conjecture using $L^2$-cohomology groups of the universal cover of $X$, leading to a stronger conjecture also involving odd-dimensional manifolds: if $X$ is a closed aspherical manifold of dimension $d$, then $b^{(2)}_k(\widetilde{X})=0$ for $k\neq d/2$, where $\widetilde{X}$ is the universal cover of $X$ (see \cite[Chapter 11]{Luck} for an overview). 
	
	%The above conjecture was also considered by Thurston for 4-manifolds (\cite[Problem 4.10]{Kirby}).	When $n=1$, \Cref{conj:Chern-Hopf-Thurston} follows from the easy fact that a Riemann surface is aspherical if and only if its genus is at least 1. 
%	If we know that $X$ admits a Riemannian metric with non-positive sectional curvature, then the conjecture follows from the Gauss-Bonnet theorem when $n=2$ (see \cite{C}). Beyond this case, the conjecture is widely open. 
	
Assuming $X$ is a complex projective manifold, several advances have been made (see \cite{Gr, JZ, CX, DCL, LMW2, AW,LP24}, among others). 	In \cite{LMW2}, it was observed that for projective manifolds, a natural generalization of asphericity is the property of having large fundamental group. Recall that a projective manifold $X$ is said to \emph{have large fundamental group} if for any irreducible subvariety $Z\subset X$, the image of the composition $\pi_1(Z_{\norm})\to \pi_1(Z)\to \pi_1(X)$ is infinite, where $Z_{\norm}$ denotes the normalization of $Z$. Furthermore, the assertion that $(-1)^n\chi(X)\geq 0$ was strengthened to 
	\begin{equation}\label{eq_statement}
		\chi(X, \sP)\geq 0 \text{\;for any perverse sheaves $\sP$ on $X$. }
	\end{equation}
	These ideas were further pursued in \cite{AW}. Although not explicitly stated, the following conjecture was expected to hold. 
	%, where in the presence of some special representation of $\pi_1(X)$, it is proved that having large fundamental group implies \eqref{eq_statement}. So 
	\begin{conj}\label{Conj:AW}
		If $X$ is a complex projective manifold with large fundamental group, then 
		$\chi(X, \sP)\geq 0$
		for any perverse sheaves $\sP$ on $X$. 
	\end{conj}
	In \cite{AW}, the conjecture was proven when there exists a semisimple, almost faithful, cohomologically rigid representation $\pi_1(X)\to \GL(r, \C)$. Here we recall that a linear representation is called \emph{almost faithful} if the kernel of the representation is finite, and it is called \emph{cohomologically rigid}, if it has no nontrivial first order deformation. 
	
	In this paper, we prove \Cref{Conj:AW} within a much broader framework: we require only the existence of    an almost faithful linear representation $\pi_1(X)\to \GL(r, K)$ over a field $K$ of arbitrary characteristic, without assuming semisimplity or cohomological ridigity of the representation.   
	
	%Let us state our main result in the strongest form. 
	
	Given a normal projective variety $X$ and a field $K$, a representation $\rho: \pi_1(X)\to \GL(r, K)$ is called \emph{large} if, for any irreducible subvariety $Z\subset X$, the composition $\pi_1(Z_{\norm})\to \pi_1(X)\xrightarrow{\rho} \GL(r, K)$ has infinite image.   
	\begin{theorem}\label{thm_main0}
		Let $X$ be a projective manifold. If there exists a large representation $\rho: \pi_1(X)\to \GL(r, K)$ where $K$ is any field, then $\chi(X, \sP)\geq 0$
		for any perverse sheaves $\sP$ on $X$. In particular, $(-1)^{\dim X}\chi(X)\geq 0$. 
	\end{theorem}	
	%	Let us mention that \cref{thm_main0} resolves a conjecture by Koll\'ar (cf. \cite[Conjecture 18.12.1]{Kol95}) in the linear case.
	
	Notice that if $X$ has large fundamental group and $\rho$ is almost faithful, then $\rho$ is large. Therefore, the following corollary follows immediately. 
	%\begin{enumerate}
	%\item  There is an  representation $\rho: \pi_1(X)\to \GL(r, \C)$, such that for any positive-dimensional irreducible subvariety $Z\subset X$, the composition $\pi_1(Z)\to \pi_1(X)\to \GL(r, \C)$ has infinite image. 
	%\item There is a linear representation $\rho: \pi_1(X)\to \GL(r, K)$ where $K$ is a field of positive characteristic, such that for any positive-dimensional irreducible subvariety $Z\subset X$, the composition $\pi_1(Z)\to \pi_1(X)\to \GL(r, K)$ has infinite image. 
	%\end{enumerate}
	%Then for any perverse sheaf $\sP$ on $X$, $\chi(X, \sP)\geq 0$. 
	%	\begin{theorem}\label{thm_main'}
		%		\footnote{added. We may incorporate it into Theorem \ref{thm_main0}}Let $X$ be a smooth projective variety. Assume that Then for any perverse sheaf $\sP$ on $X$, $\chi(X, \sP)\geq 0$. 
		%	\end{theorem}
	%If a representation $\rho: \pi_1(X)\to \GL(r, \C)$ satisfies the assumption of the theorem, we say that the associated local system $L_\rho$ is a large local system on $X$. 
	\begin{corollary}
		Let $X$ be a projective manifold with large fundamental group (e.g. $X$ is aspherical). If there exists an almost faithful representation $\rho: \pi_1(X)\to \GL(r, K)$, then for any perverse sheaf $\sP$ on $X$, $\chi(X, \sP)\geq 0$. In particular, Chern-Hopf-Thurston conjecture holds for $X$. %, that is, $(-1)^{\dim X}\chi(X)\geq 0$. Chern-Hopf-Thurston
	\end{corollary}
	
	A slightly stronger formulation of Theorem \ref{thm_main0} is given below. %, it suffices to show that 
	\begin{theorem}\label{thm_main}
		Let $X$ be a projective manifold and let  $\rho:\pi_1(X)\to \GL(r, K)$ be  a large representation, where $K$ is any field. Then, for any closed irreducible subvariety $Z\subset X$, 
		\begin{equation}\label{eq_ineq}
			T^*_Z X\cdot T^*_X X\geq 0
		\end{equation}
		where $T_Z^*X$ denotes the conormal variety of $Z$ and $\cdot$ denotes the intersection number in $T^*X$. 
	\end{theorem}
	\begin{proof}[Proof of \cref{thm_main0} assuming \cref{thm_main}]
		We need to use the notions of characteristic cycle and conormal variety which will be reviewed in \cref{sec_constr}.
		
		Since the characteristic cycle of a perverse sheaf is always effective (cf. \cref{prop_eff}), we have	\[
		CC(\sP)=\sum_{1\leq i\leq m} n_i T_{Z_i}^* X,
		\]
		where $Z_i$ are irreducible subvarieties of $X$ and $n_i\in \Z_{>0}$.  By the global index theorem (cf. \cref{thm_index}) and \cref{thm_main}, we have 
		\[
		\chi(X, \cP)=CC(\cP)\cdot T^*_X X=\sum_{1\leq i\leq m} n_i T_{Z_i}^* X\cdot T_X^*\geq 0.  \qedhere
		\] 
	\end{proof}
	Notice that since $\pi_1(X)$ is finitely generated, having a large representation $\pi_1(X)\to \GL(r, K)$ for a field $K$ of characteristic zero is equivalent to having a large representation $\pi_1(X)\to \GL(r, \C)$. Therefore, in \cref{thm_main}, we can assume either $\mathrm{char}(K)>0$ or $K=\C$. 
	
	%\begin{remark}
	%We consider the conclusion of Theorem \ref{thm_main} as certain topological positivity of $T^*X$, which depends on its symplective structure. Indeed, by \cite{DPS}, we know that  \eqref{eq_ineq} holds for any projective manifolds with nef cotangent bundle.  %We consider the conclusion of Theorem \ref{thm_main} as certain topological positivity of $T^*X$, which depends on its symplective structure. One may hope to further strengthen the conclusion to $T^*X$ being nef.\footnote{By \cite{DPS}, if $T^*X$ is nef, then the intersection number of $T^*_X X$ and any conic subvariety of $T^*X$ with complementary dimension is non-negative.} 
	%However, there are examples whose fundamental groups admit large linear representation, while its cotangent bundle is not nef. 
	%In \cite[Corollary 2]{Yiyu}, Yiyu Wang constructed some ramified covers of abelian varieties which do not have nef cotangent bundles, while such ramified covers a general rank one complex representation of the fundamental group is large.
	%\end{remark}
	
	\subsection{Sketch of the proof} 
	We will use techniques from the linear Shafarevich conjecture presented in \cite{DY23,DY24,Eys04,EKPR12} to establish Theorem \ref{thm_main}. As in \cite{AW},   Arapura and the second-named author proved  \cref{thm_main} assuming $\rho$  underlies a complex variation  of Hodge structure  ($\C$-VHS) with discrete monodromy. When $\rho$ does not satisfy this property,   Yamanoi and the first-named author demonstrated in \cite{DY23}  that that nontrivial multivalued one-forms can be produced on $X$. Our approach iteratively utilizes both multivalued one-forms and the period maps   of $\bC$-VHS to deduce the inequality \eqref{eq_ineq}. 
	
	Let us briefly  explain our strategy. 	First, we observe that if there is a $d$-valued one-form whose image in $T^*X$ intersects $T^*_Z X$ at finitely many points, then the inequality \eqref{eq_ineq} holds, since $d$ times the left-hand side is equal to the number of intersection points counting multiplicity.  
	
	In general, we introduce a vanishing cycle functor $\Phi_\eta$ of a $d$-valued one-form $\eta$ acting on the free abelian group generated by conormal varieties. We show that $d\, T^*_Z X\cdot T^*_X X=\Phi_\eta(T^*_Z X)\cdot T^*_X X$. Using $\Phi_\eta$, we can reduce inequality \eqref{eq_ineq} to the same inequality with $Z$ replaced by smaller subvarieties, provided the restriction of $\eta$ to the smooth locus of $Z$ is non-trivial.  The proof of \cref{thm_main} is divided into the following three cases. 
	
	\textbf{Case (i):} when $\mathrm{char}(K)>0$. In this case, we can prove that the above reduction process continues until $Z$ becomes a set of points. When $Z$ is a point, the inequality \eqref{eq_ineq} is obvious. 
	
	\textbf{Case (ii):} when $K=\C$ and $\rho$ is semisimple. In this case, when the above reduction process terminates, $Z_{\norm}$ underlies a   $\C$-VHS with large monodromy representation and discrete monodromy group. As proved in \cite{AW},   $Z_{\norm}$ then admits a finite morphism to the period domain of such $\bC$-VHS. The inequality \eqref{eq_ineq} then can be deduced from the curvature properties of the period domain. 
	
	\textbf{Case (iii):} when $K=\C$ and $\rho$ is not semisimple. In this case, when the reduction process terminates, $Z_{\norm}$ admits a real variation of mixed Hodge structure ($\bR$-VMHS), where its mixed period map $Z_{\rm norm}^{\rm univ}\to \cM$  has discrete fibers. In this case,  two additional difficulties arise compared to Case (ii):
	\begin{itemize}
		\item  The monodromy group may not act on the mixed period domain discretely, and thus we cannot take the quotient of the period map by the monodromy group. 
		\item  The mixed period domain does not have the desired non-positive curvature as in the pure case. %hence the computation in \cite{AW} for the pure case does not work here. 
	\end{itemize}
	These difficulties are resolved by a technical result (see Proposition \ref{prop_E}) which enables us to simultaneously explore the geometry of the mixed period map of the $\bR$-VMHS and the period map of the $\C$-VHS corresponding to  the semi-simplification of the $\bR$-VMHS. 
	
	In \cite{AW}, the discreteness of the monodromy group is obtained from a deep theorem of Esnault-Groechenig (\cite{EG18}, see also \cite{Esnault} for a survey). % based on both arithmetic and geometric Langlands programs. 
	In this paper, the discreteness of monodromy groups is deduced from nonarchemidean Hodge theory, which is same as the approach in \cite{BKT}.
	
	In theory, we could bypass Case (ii) and directly prove the more general Case (iii). However, we chose to present the proof in separate cases to isolate the technical arguments and to accommodate readers who are satisfied with understanding the proof up to Case (ii).  
	
	\subsection{Relation with the Shafarevich conjecture and hyperbolicity}
	The Shafarevich conjecture predicts that the universal cover of a complex  normal projective variety $X$ is holomorphically convex. In particular, if $X$ has large fundamental group, then its universal cover is conjectured to be Stein. Currently, this conjecture is proved for the following cases,
	\begin{enumerate}
		\item when $X$ is smooth and $\pi_1(X)$ admits a faithful linear representation over $\C$, by Eyssidieux et. al. in \cite{EKPR12};
		\item when $X$ is not necessarily smooth and $\pi_1(X)$ admits a faithful reductive linear representation over $\C$, by the first-named author, Yamanoi and Katzarkov in \cite{DY23};
		\item when $X$ is a normal surface and $\pi_1(X)$ admits an almost faithful  linear representation over a field of positive characteristic in \cite{DY24}. 
	\end{enumerate}
	
	The proofs of the Shafarevich conjecture utilize non-abelian Hodge theories both in the archimedean setting by Simpson \cite{Simpson} and in the non-archimedean setting by Gromov-Schoen \cite{GS92}.  
	
	In this paper, we do not use any results about holomorphic convexity or Steinness. However, we do employ both archimedean and non-archimedean non-abelian Hodge theories, as used in the proofs of the Shafarevich conjectures. These techiniques are robust tools for studying the Shafarevich conjecture, the hyperbolicity of algebraic varieties (cf.  e.g. \cite{Yam10,CDY22,DY24}), and the Chern-Hopf-Thurston conjecture in the linear case. Our main result \cref{thm_main} involves a specific positivity property of the cotangent bundle of $X$, which is known to be related to the hyperbolicity of $X$. For example, a projective manifold with ample cotangent bundle is Kobayashi hyperbolic. %Similarly, having a Stein universal cover is also related to hyperbolicity. For example, if the Stein universal cover of $X$ is a bounded domain, then $X$ is automatically Kobayashi hyperbolic. 
	
	Although our result relies on $X$ to have large fundamental group, we hope that positivity \eqref{eq_ineq} can be directly related to hyperbolicity. For example, does inequality \eqref{eq_ineq} always hold for irreducible subvarieties $Z$ in a Kobayashi hyperbolic projective manifold $X$? %In fact, we do not know any example of Kobayashi hyperbolic projective manifold $X$, such that \eqref{eq_ineq} fails to be strictly inequality. 
	In general, it is also interesting to find sufficient conditions on a representation of $\pi_1(X)$ such that inequality \eqref{eq_ineq} is always strict inequality.  In \cite{CDY22,DY24}, it was proved that when the Zariski closure of a (generically) large representation $\varrho:\pi_1(X)\to \GL(r,K)$ is a semisimple algebraic group, then $X$ is ``almost'' hyperbolic. This condition is sharp, as the case of abelian varieties must be excluded. Hence we conjecture that  inequality \eqref{eq_ineq} is strict if the Zariski closure of $\varrho({\rm Im}[\pi_1(Z_{\rm norm})\to \pi_1(X)])$ is a  semisimple algebraic group for every positive-dimensional subvariety $Z$ of $X$. %This supports our previous conjecture since such $X$ is Kobayashi hyperbolic by \cite{CDY22,DY24}.
	In this case, $X$ is Kobayashi hyperbolic (see \cite{Yam10,CDY22,DY24}).

	Recently, various generalizations of the Chern-Hopf-Thurston and the Singer conjectures for algebraic varieties have been formulated, e.g., in the singular setting \cite{Maxim}, and in the coherent setting \cite{AMW}. Here, we propose Chern-Hopf-Thurston type conjectures for quasi-projective varieties. For example, we conjecture that if the universal cover of a quasi-projective manifold $X$ is a bounded symmetric space, then $\chi(X, \sP)\geq 0$ for any perverse sheaf on $X$ (with respect to an algebraic stratification). A special case of this conjecture is proved in the recent work \cite{AP}. 

	\subsection*{Acknowledgement}
	This paper builds upon previous collaborations of the second author with Yongqiang Liu, Laurentiu Maxim, and Donu Arapura, as well as the collaboration of the first author with Katsutoshi Yamanoi.  We are grateful for the numerous enlightening conversations that have played a crucial role in the development of this work. We thank J\"org Sch\"urmann for carefully reviewing an earlier draft of this paper and providing valuable feedback. We also thank H\'el\`ene Esnault for patiently answering our questions. Additionally, we have benefited greatly from discussions with Jie Liu, Qizheng Yin, and Xiping Zhang.
The second author thanks Peking University and BICMR for the generous hospitality and support during the writing of this paper. The first author is partially supported by the
French Agence Nationale de la Recherche (ANR) under reference ANR-21-CE40-0010 (KARMAPOLIS).	 The second author is partially supported by a Simons fellowship. %, whose insights and perspectives have enriched our understanding of the subject matter.
	\subsection*{Notation and Convention.} 
	\begin{itemize}
		\item For any analytic/algebraic variety $Z$, we denote by $Z_{\rm norm}$ its normalization, by $Z_{\reg}$ the smooth locus of $Z$, and by $\widetilde{Z}^{\univ}$ its universal cover. 
		\item For any algebraic variety $X$, any field $K$ and any representation $\sigma:\pi_1(X)\to \GL(r,K)$, $L_\sigma$ denotes the corresponding local system. 
		\item Let $\sigma:\Gamma\to \GL(r,\bC)$ where $\Gamma$ is a finitely generated group. We denote by $\sigma^{ss}:\Gamma\to \GL(r,\bC)$   its semisimplification. 
	\end{itemize}
	
	\section{Preliminaries}
	We begin by reviewing some results about constructible functions, constructible sheaves, characteristic cycles and the global index formula. The recent survey article \cite{MS} gives a more detailed summery of these topics. 
	\subsection{Constructible functions}	
	Let $X$ be an analytic variety. A function $\gamma: X\to \Z$ is called constructible, if there exists a locally finite stratification of $X$ into locally closed smooth subvarieties $X=\bigsqcup_{i\in I}X_i$, such that the restriction of $\gamma$ to each $X_i$ is constant. 
	The Euler characteristic 
	\[
	\chi(\gamma)=\sum_{i\in I} \gamma(X_i) \cdot \chi(X_i)
	\]
	is well-defined when the above stratification can be chosen to be finite and each stratum is homotopy equivalent to a finite CW-complex. For example, this holds when $X$ is compact or when $X$ is an algebraic variety and each stratum is a locally closed algebraic subvariety. 
	
	If $f: Y\to X$ is a holomorphic map between analytic varieties, then the pullback $f^*(\gamma)\coloneqq \gamma\circ f$ is a constructible function on $Y$. 
	
	If $g: X\to Y$ is a proper holomorphic map between complex analytic varieties, then the pushforward of a constructible function $g_*(\gamma)$ defined by 
	\[
	g_*(\gamma)(y)=\chi\left(g^{-1}(Y), \gamma|_{g^{-1}(y)}\right)
	\]
	is a constructible function. Moreover, we have $\chi(Y, g_*(\gamma))=\chi(X, \gamma)$ when $Y$ is compact or when $g$ is a regular map of algebraic varieties and the constructibility is algebraic. 
	
	Given a constructible complex $\sF$ defined over a field, we can define its stalkwise Euler characteristic function $\chi_{st}(\sF)$ by 
	\[
	\chi_{st}(\sF)(x)=\chi(\sF_x).
	\]
	
	%	Given a construction function $\gamma$ in $X$, we define the \emph{support} of $\gamma$, denoted by $\Supp(\gamma)$, to be the closure of $\{x\in X\mid \gamma(x)\neq 0\}$. Then $\Supp(\gamma)$ is always a closed analytic subset of $X$. 
	
	\subsection{Constructible sheaves and characteristic cycles} \label{sec_constr}
	Let $X$ be a complex manifold. Let $D_c^b(X, K)$ be the derived category of $K$-constructible complexes,%\footnote{If $X$ is an algebraic variety, then we use algebraic constructibility. Otherwise, we use analytic constructibility.} 
	and let $Perv(X)$ be its subcategory of perverse sheaves. We are interested in the Euler characteristics of perverse sheaves, e.g., $\Q_X[\dim X]$. The Euler characteristic of a constructible complex can be computed via characteristic cycles. 
	
	%Recall that the cotangent bundle $T^*X$ has a natural holomorphic symplectic structure. 
	Given any irreducible analytic subvariety $Z$ of $X$, its conormal variety $T^*_Z X$ is defined to be the closure of the conormal bundle $T^*_{Z_{\reg}}X$ in $T^*X$. In particular, the conormal variety $T^*_X X$ of $X$ itself is equal to the zero section of $T^*X$. Every conormal variety $T^*_Z X$ is conic and Lagrangian in $T^*X$, and conversely, every conic Lagrangian subvariety of $T^*X$ is a conormal variety. Denote by $L(X)$ the abelian group of locally finite conic Lagrangian cycles on $T^*X$. The characteristic cycle is a group homomorphism
	\[
	CC: K_0\big(D_c^b(X, K)\big)\to L(X),
	\]
	where $K_0(D_c^b(X, K))$ is the Grothendieck group of $D_c^b(X, K)$. 
	We will not go into its precise definition here (see e.g., \cite[Definition 4.3.19]{Di} and \cite[Chapter IX]{KS}). Instead, we will focus on the following two important properties. 
	\begin{theorem}[\cite{Kas}, {\cite[Theorem 4.3.25]{Di}}] \label{thm_index}
		Let $X$ be a complex manifold, and let $\sF$ be a constructible complex on $X$ with compact support. Then
		\[
		\chi(X, \sF)=CC(\sF)\cdot T^*_X X
		\]
		where the right-hand side denotes the intersection number in $T^*X$. 
	\end{theorem}
	Note that even though $T^*X$ is not compact, the intersection number is well-defined because it can be interpreted as a zero cycle in the support of $\sF$ in $X$ (see  \cite[Chapter 6]{Fulton}).

	%When $Z$ is a smooth subvariety of $X$, we have  $CC(\Q_Z)=(-1)^{\dim Z}T^*_Z X$. Equivalently, the characteristic cycle of the perverse sheaf $\Q_Z[\dim Z]$ satisfies that $CC(\Q_Z[\dim Z])=T^*_Z X$. More generally, we have the following well-known result. 
	\begin{proposition}[{\cite[Corollary 5.2.24]{Di}}]\label{prop_eff}
		For any perverse sheaf $\sP$ on $X$, $CC(\sP)$ is effective. In other words, 
		\[
		CC(\sP)=\sum_{i\in I} n_i T_{Z_i}^* X,
		\]
		where $Z_i$ are irreducible subvarieties of $X$, $n_i\in \Z_{>0}$ and the sum is locally finite. 
	\end{proposition}
	
	Given a constructible complex $\sF$ on a complex manifold $X$, we can define its stalkwise Euler characteristic function $\chi_{st}(\sF): X\to \Z$ by 
	\[
	\chi_{st}(\sF)(x)=\chi(\sF_x).
	\]
	
	Denote the abelian group of $\Z$-valued constructible functions on $X$ by $F(X)$. Then $\chi_{st}$ defines a group homomorphism
	\[
	\chi_{st}: K_0\big(D_c^b(X, K)\big)\to F(X).
	\]
	Moreover, the characteristic cycle $CC: K_0\big(D_c^b(X, K)\big)\to L(X)$ factors through $\chi_{st}$. By abusing notations, we also use $CC$ to denote induced group homomorphism
	\[
	CC: F(X)\to L(X). 
	\]
	%This map is an isomorphism (\cite[Theorem 4.1.38]{Di}). 
	
	Given an irreducible subvariety $Z$ of $X$, the local Euler obstruction function $Eu_Z$ is a constructible function on $Z$ (or on $X$ with value 0 outside $Z$) uniquely characterized by the property that
	\[
	CC(Eu_Z)=(-1)^{\dim Z}T^*_Z X.
	\]
	It turns out that $Eu_Z$ is an intrinsic invariant of $Z$, which does not depend on the embedding to a smooth variety (see e.g. \cite[Page 100-102]{Di} and \cite[Sections 3,4]{Massey2}). 
	Moreover, as $Z$ varies through all irreducible analytic subvarieties of $X$, $Eu_Z$ form a basis of $F(X)$. 
	%Moreover, it is 
	%Over the smooth locus of $Z$, $Eu_Z$ has value 1. In particular, when $Z$ is smooth, $Eu_Z=1_Z$. 
	
	The above global index formula implies that
	\begin{equation}\label{eq_Eu}
		\chi(Eu_Z)=(-1)^{\dim Z} T^*_Z X\cdot T^*_X X. 
	\end{equation}
	
	The advantage of working with $Eu_Z$ and $(-1)^{\dim Z}\chi(Eu_Z)$ instead of $T^*_Z X$ and $T^*_Z X\cdot T^*_X X$ is that the former does not need to involve an embedding. 
	%The definitions of local Euler obstruction and characteristic cycle generalize to analytic varieties. So we can make the following definition. 
	\begin{definition}
		A $\Z$-constructible function $\gamma$ on an analytic variety	$X$ is called \emph{CC-effective} if it is of the following form
		\begin{equation}\label{eq_gamma}
			\gamma=\sum_{i\in I} (-1)^{\dim Z_i} n_i Eu_{Z_i},
		\end{equation}
		where $Z_i$ are irreducible subvarieties of $X$, $n_i\in \Z_{>0}$ and the sum is locally finite. 
	\end{definition}
	When $X$ is smooth, $\gamma$ is CC-effective if and only if $CC(\gamma)$ is an effective cycle in $T^* X$. 
	\begin{lemma}\label{lemma_local}
		A $\Z$-constructible function $\gamma$ on an analytic variety $X$ is CC-effective if and only if for any $x\in X$ and a small open neighborhood $U_x\subset X$ of $X$, the restriction $\gamma|_{U_x}$ is CC-effective. 
	\end{lemma}
	\begin{proof}
		Since $Eu_Z$ form a basis of $F(X)$, for any constructible function $\gamma$ there is a unique expression \eqref{eq_gamma} with $n_i\in \Z$. Since the local Euler obstruction function is a local invariant, $Eu_{Z}|_{U_x=}Eu_{Z_i\cap U_x}$. Hence,
		\[
		\gamma|_{U_x}=\sum_{i\in I, \; Z_i\cap U_x\neq 0} (-1)^{\dim Z_i} n_i Eu_{Z_i\cap U_x},
		\]
		which is also the unique expression of $\gamma|_{U_x}$ in terms of local Euler obstruction functions. Therefore, $\gamma$ is CC-effective if and only if all $\gamma|_{U_x}$ are CC-effective. 
	\end{proof}
	
	The following proposition is proved in the algebraic setting in \cite[Proposition 7.2 (2)]{AMSS}. Nevertheless, the argument also applies to the analytic setting. 
	\begin{proposition}[Aluffi-Mihalcea-Schürmann-Su]\label{prop_AMSS}
		Let $p: Z'\to Z$ be a finite morphism of irreducible analytic varieties. Then $p_*((-1)^{\dim Z'}Eu_{Z'})$ is CC-effective. More generally, if $\gamma$ is a CC-effective constructible function on $Z'$, then $p_*(\gamma)$ is also CC-effective. 
		%the pushforward of any CC-effective constructible function on $Z'$ is also CC-effective. 
	\end{proposition}
	
	\subsection{A proper pushforward formula}\label{sec_push}
	Given a constructible function $\gamma$ on a complex manifold $X$, we have considered $CC(\gamma)$ as an analytic $\dim(X)$-cycle in $T^*X$. To study the functorial properties of the characteristic cycles, we also need to consider them as Borel-Moore homology classes in an appropriate subspace of $T^*X$. 
	
	Assume $\gamma$ is constructible with respect to a Whitney stratification $\mathscr{S}$ of $X$. We define the \emph{conormal space} of $\mathscr{S}$ to be 
	\[
	T^*_{\mathscr{S}}X=\bigcup_{S\in \mathscr{S}}T^*_{\overline{S}}X,
	\]
	where $\overline{S}$ is the closure of the stratum $S$ in $X$. Then $T^*_{\mathscr{S}}X$ is a locally finite union of conic Lagrangian subvarieties. The analytic cycle $CC(\gamma)$ is supported in $T^*_{\mathscr{S}}X$, and it represents a class in $H^{BM}_{2\dim X}(T^*_{\mathscr{S}}X, \Z)$. 
	For the rest of this section, we identify a characteristic cycle as the Borel-Moore homology class it represents. 
	
	Let $f: X\to Y$ be a proper holomorphic map between complex manifolds. Let $\gamma$ be a constructible function on $X$. We review a formula to compute $CC(f_*(\gamma))$. 
	
	By the theorem in \cite[Part I, Section 1.7]{GM}, there exist Whitney stratifications $\mathscr{S}$ and $\mathscr{S}'$ of $X$ and $Y$ respectively, satisfying
	\begin{enumerate}
		\item $\gamma$ is constructible with respect to $\mathscr{S}$, and
		\item for any stratum $S\in \mathscr{S}$, there exists a stratum $S'\in \mathscr{S}'$ such that $f$ induces a submersion $S\to S'$. 
	\end{enumerate}
	
	The map $f$ induces
	\[
	T^*X\xleftarrow{u_1} f^*T^*Y\xrightarrow{u_2} T^*Y \quad \text{and}\quad T_{\mathscr{S}}^*X\xleftarrow{u_1} u_1^{-1}T_{\mathscr{S}}^*X\xrightarrow{u_2} T_{\mathscr{S}'}^*Y.
	\]
	\begin{theorem}[{\cite[Théorème 2.2]{Sab}, see also \cite[Proposition10.3.46]{MS}\label{thm_Sabbah}}]
		Under the above notations,
		\begin{equation*}%\label{eq_CC}
			CC(f_*(\gamma))=u_{2*}u_1^*(CC(\gamma)),
		\end{equation*}
		where $CC(\gamma)\in H^{BM}_{2\dim X}(T^*_{\mathscr{S}}X, \Z)$, 
		$u_1^*(CC(\gamma))\in H^{BM}_{2\dim Y}(u_1^{-1}T_{\mathscr{S}}^*X, \Z)$, 
		and $u_{2*}u_1^*(CC(\gamma))\in H^{BM}_{2\dim Y}(T_{\mathscr{S}'}^*Y, \Z)$.
	\end{theorem}
	Here, $u_{2*}$ is the proper pushforward of Borel-Moore homology, and $u_1^*$ is defined via Poincar\'e duality as the pullback of local cohomology (see \cite[Page 734]{MS} for more details). 
	%More precisely, 
	%\begin{equation}\label{eq_BM}
	%\Lambda\in H_{2d}^{BM}(\Lambda, \Z), \;\; u_1^*(\Lambda)\in H_{2b}^{BM}(u_1^{-1}(\Lambda), \Z),\;\; \text{and}\;\; u_{2*}u_1^*(\Lambda)\in H_{2b}^{BM}(u_{2}u_1^{-1}(\Lambda), \Z)
	%\end{equation}
	%where $b$ is the complex dimension of $B$ and a cycle is also identified with the underlying space when defining homology groups. 

	\subsection{Multivalued  one-forms}
	%\subsection{Multivalued one-forms and the Shafarevich morphism}
	%\textcolor{red}{How about we give some brief introduction to the Shafarevich map at the end of this subsection?}
	\begin{definition}[Multivalued  one-form]\label{def:multivalued}
		Let $X$ be a complex manifold and let $E$ be a  holomorphic  vector bundle on $X$.  A \emph{multivalued section} $\eta$ on $X$ is   a formal sum $ \sum_{i=1}^{m}n_i\Gamma_i$, where each $n_i\in \mathbb{Z}_{>0}$ and each $\Gamma_i$  is an irreducible closed subvariety of $E$ such that  
		the natural projection $\Gamma_i\to X$ is a finite surjective morphism. 
		Such multisection $\eta$ is called \emph{$d$-valued} for $d=\sum_{i=1}^m n_i\deg[\Gamma_i: X]$. 
		We say that $\eta$ is \emph{reduced} if all $\Gamma_i$ are distinct, and $\eta$ is \emph{irreducible} if $m=1$ and $n_1=1$. 
		We say that $\eta$ is \emph{trivial} if $m=1$ and $\Gamma_1$ is the zero section of $E\to X$.   Multivalued sections of $T^*X$ will be called \emph{multivalued (holomorphic) one-forms}. 
	\end{definition}
	
	Let $\eta$ be a multivalued one-form on a complex manifold $X$. We sometimes denote by $\Gamma_\eta$ instead of $\eta$ to emphasize it is an analytic cycle in $T^*X$. 
	By definition, there exists a largest Zariski open subset $X^\circ\subset X$ such that $\Gamma_\eta\cap T^*X^{\circ}\to X^{\circ}$ is \'etale. We call $X^\circ$ to be the \emph{unbranching locus} of $\eta$ and the complement $X\setminus X^{\circ}$ the \emph{branching locus}.  In this case, for any $x\in X^\circ$, there exist  an open neighborhood of $U$ of $x$ and holomorphic one-forms $\eta_1,\ldots,\eta_m$ on $U$ such that $\eta|_{U}$ is equal to the union of $\eta_1,\ldots,\eta_m$.  Therefore, we say that $\eta$ is \emph{closed} if for any point $x\in X^\circ$, each of the above $\eta_i$ is a closed one-form.
	
	\begin{remark}\label{remark_restriction}
		Let $Y$ be a locally closed complex submanifold of $X$. Given any multivalued one-form $\eta$ of $X$, using the pullback map $T^*X|_{Y}\to T^*Y$ we can define the restriction multivalued one-form $\eta|_{Y}$. Clearly, when $\eta$ is closed, so is $\eta|_Y$. 
	\end{remark}

	\begin{lemma}\label{lemma_closed}
		Let $X$ be a projective manifold. Every multivalued one-form $\eta$ on $X$ is closed. 
	\end{lemma}
	\begin{proof}
		Without loss of generality, we can assume that $\eta$ is irreducible. Since $X$ is projective and the natural map $\Gamma_\eta\to X$ is finite, $\Gamma_\eta$ is a projective variety. Let $Y\to \Gamma_\eta$ be a resolution of singularity, which is an isomorphism over the smooth locus of $\Gamma_\eta$. Let $\theta$ be the tautological holomorphic one-form on $T^*X$, and let $\theta_Y$ be the pullback of $\theta$ via the composition $Y\to \Gamma_\eta\to T^*X$. Since $Y$ is a projective manifold, by Hodge theory $\theta_Y$ must be closed. Over the unbranching locus of $\eta$, the composition $Y\to \Gamma_\eta\to X$ is \'etale, and the pushforward of $\theta_Y$ to $X$ is equal to $\eta$. Thus, $\eta$ is closed. 
	\end{proof}
	By the work of \cite{Eys04} and \cite{CDY22}, we can construct non-trivial multivalued one-forms from unbounded representations $\pi_1(X)\to \GL(r,K)$ where $K$ is a non-archimedean local field.    
	\begin{proposition}\label{thm:KE}
		Let $X$ be a smooth projective variety and let $\rho:\pi_1(X)\to \GL(r,K)$ be a  reductive representation where $K$ is a non-archimedean local field.  Then there exists a proper surjective morphism $s_\rho:X\to S_\rho$ to a normal projective variety with connected fibers and a closed multivalued one-form $\eta_\rho$ such that  for any irreducible closed subvariety $Z\subset X$, the following properties are equivalent: 
		\begin{enumerate}
			\item\label{item:KZ1} $\rho({\rm Im}[\pi_1(Z_{\rm norm})\to \pi_1(X)])$ is bounded;
			\item \label{item:KZ3} $\rho({\rm Im}[\pi_1(Z)\to \pi_1(X)])$ is bounded;
			\item \label{item:KZ2}  $s_\rho(Z)$ is a point;
			\item  \label{item:KZ4}  the restriction  $\eta_\rho|_{Z_{\reg}}$ is trivial, where $Z_{\rm reg}$ is the smooth locus of $Z$.
		\end{enumerate}  
		In particular, if $\rho(\pi_1(X))$ is unbounded, then $\eta_\rho$ is non-trivial.  
	\end{proposition}
	We will call the above map $s_\rho$ the (Katzarkov-Eyssidieux) reduction map for $\rho$.  The construction  of the multivalued holomorphic one-form $\eta_\rho$ associated with $\rho$ can be found in \cite[Step 2 in the proof of Theorem H]{CDY22}. The closedness of $\eta_\rho$ follows from \cref{lemma_closed}. The equivalence of Item \ref{item:KZ1} and  Item \ref{item:KZ2} is proved by Katzarkov \cite{Kat97} and Eyssidieux \cite{Eys04}.   The equivalence of the first three items can be found in \cite[Theorem H]{CDY22}.  Hence we only need to prove their equivalence to Item \ref{item:KZ4}. % and the closedness of $\eta_\rho$. 
	\begin{proof}[Proof of the equivalence of Items \ref{item:KZ2} and \ref{item:KZ4}]
		To simplify the notation, we write $\eta$ instead of $\eta_\rho$.   By \cite[Definition 5.11]{CDY22}, we can find a finite Galois cover $f:Y\to X$  of Galois group $G$ (so-called \emph{spectral covering}) from a normal projective variety such that $f$ is \'etale outside the branching locus of $\Gamma_{\eta}\to X$, and  $f^*\eta$ becomes single-valued, i.e. there exist sections $\{\omega_1,\ldots,\omega_m\}\subset H^0(Y,f^*\Omega_X^1)$ such that $f^*\eta=\{\omega_1,\ldots,\omega_m\}$. Let $a:Y\to A$ be the \emph{partial Albanese map} associated with $\{\omega_1,\ldots,\omega_m\}$  (cf. \cite[Definition 5.19]{CDY22}), where $A$ is an abelian variety. By \cite[Claim 5.15]{CDY22}, $\{\omega_1,\ldots,\omega_m\}$ is invariant under $G$. Moreover, by \cite[Step 4 in the proof of Theorem H]{CDY22},  $G$ acts on $A$ such that $a$ is $G$-equivariant. Then by \cite[Proof of Theorem H]{CDY22}, $s_\rho$ is the Stein factorization of the quotient map $X\to A/G$, as in the commutative diagram, 
		%		\begin{equation}
			%			\begin{tikzcd}
				%				Y\arrow[r, "f"] \arrow[dd,"a"] & X\arrow[d,"s_\rho"']\arrow[dd, bend left=38]\\
				%				& S_\rho\arrow[d]\\
				%				A\arrow[r] & A/G\end{tikzcd}
			%		\end{equation}
		\begin{equation}
			\begin{tikzcd}
				Y\arrow[d, "f"] \arrow[rr,"a"] && A\arrow[d, "\pi"]\\
				X\arrow[r, "s_\rho"]&S_\rho\arrow[r]& A/G.
			\end{tikzcd}
		\end{equation}
		Let $Z'$ be a connected component of $f^{-1}(Z)$. Then $Z'$ and $Z$ have the same image in $A/G$. Since $s_\rho$ is the Stein factorization of $X\to A/G$, %$a(Z')$ is a point if and only if $\pi\circ a(Z')$ is a point if and only if $s_\rho(Z)$ is a point. 
		\[
		\text{$a(Z')$ is a point}\Leftrightarrow \text{$\pi\circ a(Z')$ is a point}\Leftrightarrow \text{$s_\rho(Z)$ is a point}.
		\]
		%Then there exists a subgroup $G_o\subset G$ such that $Z'\to Z$ is a finite Galois cover with the Galois group $G_o$. 
		%Therefore, $a(Z')$ is a point if and only if $s_\rho(Z)$ is a point.  
		By \cite[Lemma 1.1]{CDY22}$, a(Z')$ is a point if and only if $\omega_i|_{Z'_{\rm reg}}= 0$ for each $i$. %where $Z'_{\rm reg}$ is the smooth locus of $Z'$.
		
		Clearly, $\omega_i|_{Z'_{\rm reg}}= 0$ for each $i$ if and only if $\eta|_{Z_{\reg}}$ is trivial.
		Thus, the equivalence between Items \ref{item:KZ2} and   \ref{item:KZ4} follows.
	\end{proof}
	
	\subsection{A factorization map}\label{sec_Stein} 
	In this subsection, we will review some constructions in  the proof of the reductive Shafarevich conjecture in \cite{DY23}. We first apply Proposition \ref{thm:KE} to construct a fibration which is essential in our proof. This construction allows us to factorize non-rigid representations into those underlying $\bC$-VHS with discrete monodromy. 
	\begin{definition}[Factorization map]\label{def:reduction ac}
		Let $X$ be a smooth projective variety. 	We fix a positive integer $r>0$.	  We define a \emph{factorization map} $s_{{\rm fac},r}:X\to S_{{\rm fac},r}$ to be the \emph{simultaneous Stein factorization} of the Katzarkov-Eyssidieux reductions   $\{s_{\tau}:X\to S_\tau\}_{\tau}$, where  $\tau:\pi_1(X)\to \GL(r,K)$ ranges over all semisimple representations with $K$  a non-archimedean local field of characteristic $0$. We refer the readers to  \cite[Lemma 1.28]{DY23} and \cite[Lemma on page 7]{Car60} for the precise definition of the simultaneous Stein factorization. In particular, $s_{{\rm fac},r}:X\to S_{{\rm fac},r}$  is a proper morphism to a normal projective variety with connected fibers such that
		\begin{enumerate}
			\item  all the above maps $s_{\tau}$ factor  through $s_{{\rm fac},r}$;
			\item for any closed subvariety $Z$ of $X$,   $s_\tau(Z)$ is a point for all the above $s_\tau$ if and only if $s_{{\rm fac},r}(Z)$ is a point.
		\end{enumerate}
	\end{definition}   
	\subsection{Shafarevich morphism} 
	We will recall the definition of the Shafarevich morphism of a representation of the fundamental group of a projective variety. 
	%enabling us to utilize \cref{prop_linear} in addressing the proof of \cref{thm_main} in the case where $\rho$ is linear and $K=\bC$. 
	% This will be our foundation of the proof of \cref{thm_main} in the case when $K=\bC$ and $\rho$ is not necessarily semisimple. 
	\begin{definition} [Shafarevich morphism]\label{def:Shafarevich}
		Let $X$ be a projective manifold.
		\begin{thmlist}
			\item Let $H$  be a normal subgroup of $\pi_1(X)$. The \emph{Shafarevich morphism} of the pair $(X,H)$ is a holomorphic map to a normal projective variety ${\rm sh}_H:X\to \mathrm{Sh}_H(X)$ with connected fibers such that for any subvariety $Z$ of $X$, $ \im[\pi_1(Z_{\norm})\to \pi_1(X)/H]$ is finite if and only if $\mathrm{sh}_H(Z)$ is a point. 
			\item The \emph{Shafarevich morphism} of  a   linear presentations of $\varrho:\pi_1(X)\to \GL(r,K)$, denoted by $\mathrm{sh}_\varrho: X\to \mathrm{Sh}_\varrho(X)$, is the Shafarevich morphism  of the pair $(X,\ker\varrho)$. 
			\item Let $M$ be a subset of the moduli space  of representations $M_{\rm B}(\pi_1(X),\GL_r)(\bC)$.   The \emph{reductive Shafarevich morphism} of $M$, denoted by $\mathrm{sh}_M: X\to \mathrm{Sh}_M(X)$, is  the Shafarevich morphism  of the pair $(X,H)$, where $H$ is the intersection of kernels of all \emph{semisimple} representations $\varrho:\pi_1(X)\to \GL(r,\bC)$ with $[\varrho]\in M$. 
		\end{thmlist}
	\end{definition}
	The Shafarevich morphism is unique if it exists. The existence of $\mathrm{sh}_M$ for various choices of $M$ is proved in \cite{Eys04,EKPR12,DY23,DY24}. Note that for if $\varrho:\pi_1(X)\to \GL(r,K)$ is a large representation, then the Shafarevich morphism ${\rm sh}_\varrho:X\to {\rm Sh}_\varrho(X)$ of $\varrho$ is the identity map.  	
	\subsection{Pure Hodge structures and period maps}\label{sec:period}
	In this subsection we briefly review the definitions of $\bC$-Hodge structures, pure period domains and period maps.   We refer the readers to \cite{CMP,SS22} for more details. 
	
	A polarized $\bC$-Hodge structure (of weight $m$) is a triple $(V=\oplus_{p+q=m}V^{p,q},S)$, where $V$ is a $\bC$-vector space together with  a decomposition $V=\oplus_{p+q=m}V^{p,q}$, and $S$ is the  \emph{polarization} that is a non-degenerate hermitian  form on $V$ such that the above decomposition is orthogonal with respect ot $S$ and $(-1)^pS|_{V^{p,q}}$ is positive-definite.  If $(V,S)$ is endowed with a real structure such that we have the complex conjugate $\overline{F^{p,q}}=F^{q,p}$, then it is called a $\bR$-Hodge structure. 
	
	The Hodge filtration is defined to be $F^p:=\oplus_{i\geq p}V^{i,m-i}$.  Fixing  $m$ and $\dim_{\bC}F^p$, the set of all such filtration $F^\bullet$  is a complex flag manifold, which is denoted by $\cD^\vee$. It is a closed submanifold of a product of Grassmannians, and   is thus a projective manifold. The \emph{period domain}, denoted by $\cD$, is the subset of all complex polarized Hodge structures are charcterized by
	\begin{enumerate}[label=(\alph*)]
		\item $
		F^p=F^p\cap (F^{p+1})^\perp\oplus F^{p+1}$. 
		\item  $(-1)^pS$ is positive definite over $F^p\cap (F^{p+1})^\perp$.
	\end{enumerate}
	It is an open submanifold of $\cD^\vee$. Since the groups ${\rm GL}(V)$ and $\GL(V,S)$ act transitively on $\cD^\vee$ and $\cD$ respectively,  $\cD^\vee$ and $\cD$ are thus   homogeneous spaces.  We also use $F^\bullet$ to denote the $\bC$-Hodge structure. 
	
	For any Hodge structure $F^\bullet\in \cD^\vee$, the holomorphic tangent space $T_{F^\bullet}\cD^\vee$ of $\cD^\vee$  at $F^\bullet$ is isomorphic to  
	\begin{align*} 
		\End(V)/ \{A\in \End(V) \mid A(F^p)\subset F^p \    \mbox{for all}\  p \}.
	\end{align*}
	For any $A\in \End(V)$, we denote by $[A]_F$ its image in $T_{F^\bullet}\cD^\vee$.  
	A tangent vector $[A]_{F^\bullet}$  in $T_{{F^\bullet}}\cD^\vee$ is called \emph{horizontal} if $A(F^p)\subset F^{p-1}$ for all $p$. %The subbundle of $T_{\cD^\vee}$ consisting of 
	All horizontal vectors form a vector subbundle of $T{\cD^\vee}$, which we denote by $T^{h}{\cD^\vee}$. A holomorphic map $f:\Omega\to \cD^\vee$ from a complex manifold $\Omega$ is called \emph{horizontal} if $df: T\Omega\to f^*T{\cD^\vee}$ factors through $f^*T^{h}{\cD^\vee}$. 
	
	A \emph{$\bC$-variation of Hodge structure} ($\C$-VHS for short)  on a complex manifold $X$ is a family of polarized $\C$-Hodge structures on $X$ subject to a Griffiths transversality condition (see e.g., \cite{Simpson, SS22} for more details). 
	Given a $\C$-VHS with monodromy representation $\varrho:\pi_1(X)\to \GL(V,S)$, it  induces a $\varrho$-equivariant horizontal holomorphic map $\phi:\widetilde{X}^{\rm univ}\to \cD$, called the \emph{period map}. The image $\varrho(\pi_1(X))$ is called the \emph{monodromy group}.
	
	\subsection{Mixed Hodge structures and mixed period maps}\label{sec:mixed period}
	We recall the definition of $\bR$-mixed Hodge structures of weight length one and their mixed period maps. We refer the readers to \cite{Pea00,Her99,Car87} for more details. 
	
	%We describe classifying spaces $\cM$ for polarized $\mathbb{R}$-MHS of weight length 1.  Let $H$  be a real finite dimensional vector space.
	A  graded  polarized $\bR$-mixed Hodge structure  of length 1 is quadruple $(V_\bR,W_\bullet,F^\bullet,S_i)$ consisting of 
	\begin{itemize}
		\item A  real finite dimensional vector space $V_{\bR}$;
		\item  an increased (weight) filtration $\{0\}=W_{-2}\subset W_{-1}\subset W_0=V_{\bR}$;
		\item   a decreased (Hodge) filtration $F^\bullet $ of $V$, where $V:=V_{\bR}\otimes_\bR\bC$;
		\item  two  non-degenerate hermitian form $S_{-1}$ and $S_0$ on the graded quotients  ${\rm Gr}_{-1}^WV_\bR$ and ${\rm Gr}_0^WV_\bR$ of $V_\bR$  respectively,
	\end{itemize} 
	such that   $ {\rm Gr}_m^WV_\bR$ carries a pure $\bR$-Hodge structure of weight $m$ polarized by $S_m$. Here the Hodge filtration $F^\bullet{\rm Gr}_m^WV$ is given by    
	$$F^p{\rm Gr}_m^WV:=\frac{F^p\cap W_m\otimes \bC}{F^{p}\cap W_{m-1}\otimes\bC}.$$ 
	We fix $(V_\bR,W_\bullet,S_i)$.  After fixing  $\dim_{\bC}F^p{\rm Gr}_m^WV$, the \emph{mixed period domain} $\cM$ is the set of polarized $\bR$-mixed Hodge structures on $(V_\bR,W_\bullet)$, i.e. the set of decreasing
	filtrations  $F^\bullet $ such that  $(V_\bR,W_\bullet,F^\bullet,S_i)$  is a mixed Hodge structure.  %The set of mixed Hodge structures which satisfy only the first Riemann bilinear relations relative to the polarization, will be denoted $\cM^{\vee}$.  
	
	Given an $\R$-mixed Hodge structure  $(V_\bR,W_\bullet,F^\bullet,S_i)$ in $\cM$, %let $(F^\bullet\mathrm{Gr}_{-1}^WV,F^\bullet\mathrm{Gr}_0^WV)$ be the ordered set of pure $\bR$-Hodge structures defined on the graded quotients of $W_\bullet$. 
	since each $(\operatorname{Gr}_k^WV_\bR,F^\bullet\mathrm{Gr}_{k}^WV,S_k)$ is classified by a pure period domain $\cD_k$ for $k=-1,0$, the graded quotient   $ \mathrm{Gr}_\bullet^WV_\bR$ of $W_\bullet$ is then classified by a point of
	$$
	{\rm Gr}^W \cM:=\cD_{-1}\times \cD_0.
	$$ 
	Thus, we have a natural projection
	$
	\pi: \cM \to {\rm Gr}^W\cM,
	$ 
	which is a holomorphic map between complex manifolds.% and $\pi$ is a holomorphic map. 
	
	Write $V_{k,\bR}:=\operatorname{Gr}_k^WV_\bR$. Let $\GL(V_\bR)^W$ be the  real subgroup of $\GL(V_\bR)$ preserving the weight filtration $W_\bullet$.    Then we have a natural homomorphism
	$$
	q:\GL^W(V_\bR)\to \GL(V_{-1,\bR})\times  \GL(V_{0,\bR}).
	$$
	By \Cref{sec:period}, the subgroup $\GL(V_{k,\bR},S_{k})$ of $ \GL(V_{k,\bR})$ acts on $\cD_k$ transitively. Let $G$ be the inverse image of $\GL(V_{-1,\bR},S_{-1})\times \GL(V_{0,\bR},S_{0})$ under $q$, which is a real algebraic group. Then $G(\bR)$ acts on $\cM$ and  $\pi$ is $q$-equivariant. Note that the kernel $U$ of $G\to \GL(V_{-1,\bR},S_{-1})\times \GL(V_{0,\bR},S_{0})$ is a commutative real algebraic group isomorphic to a real vector space.
	
	In the same vein as \Cref{sec:period}, we can define the horizontal bundle $T_\cM^{h}$ on $\cM$, which is a holomorphic subbundle of the tangent bundle of $\cM$. 
	
	An $\bR$-variation of mixed Hodge structure ($\bR$-VMHS for short) on a complex manifold $X$ is a family of $\R$-mixed Hodge structures subject to certain conditions (see \cite{Pea00} for more details). 
	In particular, an $\bR$-VMHS determines a monodromy representation $\varrho:\pi_1(X)\to G(\bR)$ and a $\varrho$-equivariant holomorphic map $\phi:\widetilde{X}^{\rm univ}\to \cM$ such that $f$ is horizontal, i.e., $d\phi:T{\widetilde{X}}\to f^*T^{h}\cM$, and it satisfies the property that,  $\pi\circ \phi:\widetilde{X}^{\rm univ}\to {\rm Gr}^W\cM$ is a  $q\circ\varrho$-equivariant horizontal holomorphic map, which defines an $\bR$-VHS.   Such $\phi$ is called the \emph{mixed period map}  of this $\bR$-VMHS and $\varrho(\pi_1(X))$ is called the \emph{monodromy group}.

	%Since $\cD_i^\vee$ is a complex compact projective manifold and $\cD^i\subset \cD^i$ is an open subset, it follows that $\cM=\pi^{-1}({\rm Gr}^W \cM)$ is also an open subset of $\cM^\vee$.    We denote by $\pi_0:\cM\to {\rm Gr}^W\cM$  the restriction of $\pi$ on $\cM$.
	%The structure of $\cM$ can be described as follows:
	We recollect the following standard fact about the mixed period domain (cf. \cite[p. 218]{Car87}). 
	\begin{lemma}
		Let $\cM$ be as above. Then  $\pi:\cM\to {\rm Gr}^W\cM$ is a holomorphic vector bundle with the  fiber at a point $P\in {\rm Gr}^W\cM$  being canonically isomorphic to ${\rm Hom}(V_0,V_{-1})\otimes\bC /F^0{\rm Hom}(V_0,V_{-1})\otimes\bC $. Here we denote by $V_i:={\rm Gr}^W_iV$, and    ${\rm Hom}(V_0,V_{-1})$  is endowed with a natural Hodge structure of weight $-1$ induced from $P$.    For  the kernel $U$ of the homomorphism $G\to \GL(V_{-1,\bR},S_{-1})\times \GL(V_{0,\bR},S_{0})$,    $U(\bC)$ acts on the fibers of $\pi$ as a translation.  \qed
	\end{lemma}

	\section{Nearby cycle functor of a multivalued one-form}
	In the first part of this section, we define a vanishing cycle functor of a multivalued one-form and prove some useful properties of this functor. In the second part, we prove Theorem \ref{thm_main} in the case when $K$ is a field of positive characteristic. 
	\subsection{Definition of $\Phi_\eta$}
	Given an irreducible multivalued one-form $\eta$ on a complex manifold $X$ and an irreducible conic Lagrangian cycle $\Lambda$ in $T^*X$, we define two irreducible cycles in $T^* X\times \C$: $\Lambda^\diamond=\Lambda\times \C$ and $\Gamma_\eta^{\diamond}$ is the cycle such that its restriction to $T^*X \times \{s\}$ is equal to $\Gamma_{s\eta}$. Since $\Gamma_\eta^{\diamond}$ is a multisection of $T^*X\times \C$ as a vector bundle over $X\times \C$, it follows that $\Lambda^\diamond\times_{X\times \C} \Gamma_{\eta}^\diamond$ is an $(m+1)$-cycle in $(T^*X\times \C)\times_{X\times \C}(T^*X \times \C)$, where $m=\dim X$. Denote the fiberwise addition map by 
	\[
	\mathfrak{S}: (T^*X\times \C)\times_{X\times \C}(T^*X \times \C)\to T^*X \times \C.
	\]
	%Since the restriction of $\mathfrak{S}$ to the support of $\Lambda^\diamond\times_{X\times \C} \Gamma_{\eta}^\diamond$ is a proper morphism, the pushforward $\mathfrak{S}_*(\Lambda^\diamond\times_{X\times \C} \Gamma_{\eta}^\diamond)$ is a well-defined $(d+1)$-cycle in $T^*X\times \C$, which we denote by $\Lambda_\eta^\diamond$. Taking closure in $T^*X\times \bP^1$ and keeping the multiplicities, we have $\overline{\Lambda_\eta^\diamond}$, an $(m+1)$-cycle in $T^*X\times \bP^1$. Let $\Phi_\eta\Lambda$ be the restriction of $\overline{\Lambda_\eta^\diamond}$ to $T^*X\times \{\infty\}$. 
	Since the natural projection $\Gamma_\eta^\diamond\to X\times \C$ is a finite morphism, so is the restriction map
	\[
	\mathfrak{S}: (T^*X\times \C)\times_{X\times \C}\Gamma_\eta\to T^*X\times \C.
	\]
	Thus, the restriction of $\mathfrak{S}$ to the support of $\Lambda^\diamond\times_{X\times \C} \Gamma_{\eta}^\diamond$ is also a finite morphism. The pushforward $\mathfrak{S}_*(\Lambda^\diamond\times_{X\times \C} \Gamma_{s\eta}^\diamond)$ is a $(d+1)$-cycle in $T^*X\times \C$, which we also denote by $\Lambda_\eta^\diamond$. 
	\begin{lemma}\label{lemma_ana}
		As a subspace of $T^*X\times \bP^1$, $\Lambda_\eta^\diamond$ is locally closed with respect to the analytic Zariski topology.
	\end{lemma}
	\begin{proof}
		If we consider $T^*X\times \C$ as a vector bundle over $X$, then it follows from definition that $\Lambda_\eta^\diamond$ is a conic cycle. Consider the fiberwise projective compactification 
		\[
		T^*X\times \C\subset \bP\big((T^*X\times \C)\oplus \C_{X}\big).
		\]
		Since $\Lambda_\eta^\diamond$ is conic, it is locally closed in $\bP\big((T^*X\times \C)\oplus \C_{X}\big)$ with respect to the analytic Zariski topology.
		
		Now, consider the fiberwise projective compactifications
		\[
		T^*X\subset \bP(T^*X\oplus \C_X), \quad\text{and}\quad X\times \C=\C_X\subset \bP(\C_X\oplus \C_X)=X\times \bP^1.
		\]
		As partial compactifications of $T^*X\times \C$, $\bP\big((T^*X\times \C)\oplus \C_{X}\big)$ and $\bP(T^*X\oplus \C_X)\times_{X} (X\times \bP^1)$ are birational to each other. In fact, one can easily construct blowup and blowdown maps connecting them. Since $\Lambda_\eta^\diamond$ is locally closed in the first partial compactification, it is also a locally closed in the second.
	\end{proof}
	
	Taking closure in $T^*X\times \bP^1$ and keeping the multiplicities, we have $\overline{\Lambda_\eta^\diamond}$, a $(d+1)$ cycle in $T^*X\times \bP^1$. Let $\Phi_\eta\Lambda$ be the restriction of $\overline{\Lambda_\eta^\diamond}$ to $T^*X\times \{\infty\}$. 
	
	\begin{remark}\label{remark1}
		The $(m+1)$-cycle $\Lambda^\diamond\times_{X\times \C} \Gamma_{\eta}^\diamond$ may have higher multiplicities. In fact, it is defined as the scheme-theoretic intersection of $\Lambda^\diamond\times \Gamma_{\eta}^\diamond$ and the preimage of the diagonal of $(X\times \C)\times (X\times \C)$ under the map $(T^*X\times \C)\times (T^*X\times \C)\to (X\times \C)\times (X\times \C)$. It is straightforward to check that the intersection has expected dimension, but may have higher multiplicities (see \cite[Section 7.1]{Fulton}). 
	\end{remark}

	\begin{remark}\label{remark_def}
		When $\eta$ is a single-valued one-form, $\Gamma_\eta$ is a section of $T^* X$, and $\Phi_\eta\Lambda$ is simply MacPherson's description of the deformation to normal cone (see \cite[Remark 5.1.1]{Fulton}). 
	\end{remark}
	\begin{lemma}\label{lemma_Z}
		Let $Z$ be a closed complex submanifold of $X$, and denote $T^*_ZX$ by $\Lambda$. Let $\eta$ be a multivalued one-form on $X$, and let $\eta|_Z$ be the restriction of $\eta$ to $Z$ as in \cref{remark_restriction}. Then
		\begin{equation}\label{eq_Z}
			\Phi_{\eta}(\Lambda)=u^*\left(\Phi_{\eta|_Z}T^*_Z Z\right)
		\end{equation}
		where $u: T^*X|_Z\to T^*Z$ is the pullback map and $u^*$ is the flat pullback on analytic cycles. 
	\end{lemma}
	\begin{proof}
		Let $u_0: T^*X|_Z\times \C\to T^*Z\times \C$ and $u_{1}: T^*X|_Z\times \bP^1\to T^*Z\times \bP^1$ be the product of $u$ and the identity maps on $\C$ and $\bP^1$ respectively. Then it follows from definition that $\overline{\Lambda_{\eta}^\diamond}$ is contained in $T^*X|_Z\times \bP^1$, and as analytic cycles on $T^*X|_Z\times \C$ and $T^*X|_Z\times \bP^1$ respectively, we have
		\begin{equation}\label{eq_u}
			u_0^*\left({\Gamma_{\eta|_Z}^\diamond}\right)={\Lambda_{\eta}^\diamond}\quad 
			\text{and}\quad u_1^*\left(\overline{\Gamma_{\eta|_Z}^\diamond}\right)=\overline{\Lambda_{\eta}^\diamond}.
		\end{equation}
		By the flatness of $u_1$ and the definition of $\Phi_\eta$, we have 
		\[
		u_1^*\left(\overline{\Gamma_{\eta|_Z}^\diamond}\cap T^*Z\times \{\infty\}\right)=\overline{\Lambda_{\eta}^\diamond}\cap T^*X|_Z\times\{\infty\}=\Phi_\eta(\Lambda).
		\]
		On the other hand, if we let $\Lambda_Z=T^*_Z Z$, then by definition 
		\[
		\Phi_{\eta|_Z}(\Lambda_Z)=\overline{\Lambda_{Z, \eta|_{Z}}^\diamond}\cap T^*Z\times \{\infty\} \quad\text{and}\quad \overline{\Lambda_{Z, \eta|_{Z}}^\diamond}=\overline{\Gamma_{\eta|_Z}^\diamond}. 
		\]
		Thus, the desired equality \eqref{eq_Z} follows.
	\end{proof}
	\begin{proposition}\label{prop_Lagrangian}
		Given an irreducible conic Lagrangian cycle $\Lambda=T^*_Z X$, $\Phi_\eta \Lambda$ is an effective conic cycle supported in $T^*X|_Z$. Moreover, if $\eta$ is closed, then $\Phi_\eta \Lambda$ is also Lagrangian. 
	\end{proposition}
	\begin{proof}
		By definition, $\Phi_\eta \Lambda$ is effective. %Consider the $\C^*$-action on $T^*X\times \C$ by fiberwise multiplication on $T^*X$ and scalar multiplication on $\C$. This action preserves both $\Lambda^\diamond$ and $\Gamma_\eta^\diamond$, and hence it induces a $\C^*$-action on $\Lambda_\eta^\diamond$. 
		Since the $\C^*$-action on $T^*X\times \C$ induced by the vector bundle structure (as in the proof of Lemma \ref{lemma_ana}) extends to $T^*X\times \bP^1$, the $\C^*$-action on $\Lambda_\eta^\diamond$ extends to the closure $\overline{\Lambda_\eta^\diamond}$. Thus, the restriction $\overline{\Lambda_\eta^\diamond}$ to $T^*X \times \{\infty\}$ is a conic cycle. 
		
		Since $\Phi_\eta \Lambda$ is the limit of $\Lambda_\eta^\diamond\cap (T^* X\times \{s\})$ as $s$ approaches to $\infty$, and since the limit of Lagrangian cycles is also Lagrangian, it suffices to show that $\Lambda_{\eta, s}\coloneqq\Lambda_\eta^\diamond\cap (T^* X\times \{s\})$ 
		is Lagrangian for any $s\in \C$. Since $\Lambda^\diamond_\eta$ is irreducible and of dimension $m+1$, $\Lambda_{\eta, s}$ is pure of dimension $m$. Here, we recall that $m$ is the dimension of $X$.
		
		By definition, identifying $T^* X\times \{s\}$ with $T^* X$, we have (not counting multiplicity)
		\[
		\Lambda_{\eta, s}=\big\{(x, \zeta)\mid \zeta\in \Lambda|_x+s\cdot \Gamma_\eta|_x\big\},
		\]
		where 
		\[
		\Lambda|_x+s\cdot \Gamma_\eta|_x=\big\{\alpha+s\beta\mid \alpha\in \Lambda\cap T^*_x X \;\text{and}\; \beta \in \Gamma_\eta\cap T^*_x X \big\}.
		\]
		Since the projection $\Gamma_\eta\to X$ is a finite morphism, over every point of $Z$, the fibers of the two projections $\Lambda=T^*_Z X\to Z$ and $\Lambda_{\eta, s}\to Z$ have the same dimension. Since $\Lambda$ is irreducible and $\dim \Lambda=\dim \Lambda_{\eta, s}=m$, by counting dimensions, the restriction of $\Lambda_{\eta, s}\to Z$ to every irreducible component of $\Lambda_{\eta, s}$ is dominant. 
		Thus, it suffices to show that $\Lambda_{\eta, s}$ is Lagrangian over a dense open subset of $Z$. By the first equation in \eqref{eq_u}, we have
		\[
		\Lambda_{\eta, s}|_{Z_{\reg}}=u^*\big(\Gamma_{s\eta|_{Z_{\reg}}}\big)
		\]
		where $u: T^*X|_{Z_{\reg}}\to T^*Z_{\reg}$ is the pullback map. Since locally, we can realize $X$ as the product of $Z_{\reg}$ and another complex manifold, $u^*(\Gamma_{s\eta|_{Z_{\reg}}})$ is Lagrangian in $T^*X$ if and only if $\Gamma_{s\eta|_{Z_{\reg}}}$ is Lagrangian in $T^*Z_{\reg}$. Locally over the unbranching locus of $s\eta|_{Z_{\reg}}$, $\Gamma_{s\eta|_{Z_{\reg}}}$ is a union of closed holomorphic one-forms. It is a well-known fact that the image of a closed one-form in the cotangent bundle is a Lagrangian submanifold. Since $\eta$ is a closed multivalued one-form, so is $s\eta|_{Z_{\reg}}$. Hence, over the unbraching locus, $\Gamma_{s\eta|_{Z_{\reg}}}$ is a Lagrangian submanifold of $T^*Z_{\reg}$, and we have finished the proof. 
	\end{proof}
	%	\begin{lemma}
		%		Let $X$ be a complex manifold, and $\eta$ be a closed holomorphic one-form on $X$. As a submanifold of $T*X$, $\Gamma_\eta$ is Lagrangian. 
		%	\end{lemma}
	%	\begin{proof}
		%		There is a tautological one-form $\theta$ on $T^*X$ such that the symplectic form $\omega$ on $T^*X$ is defined by $\omega=d\theta$. Considering $\eta$ as a section of $T^*X$, we have a natural map $i: X\to \Gamma_\eta$. By the definition of $\theta$, $i^*(\omega|_{\Gamma_\eta})=d\eta=0$. Thus, $\omega|_{\Gamma_\eta}=0$, that is, $\Gamma_\eta$ is Lagrangian. 
		%By definition, $\mathfrak{T}_\eta^*(\theta)=\theta+\pi^*{\eta}$, where $\pi: T^*X\to X$ is the bundle map. Thus, 
		%		\[
		%		\mathfrak{T}_\eta^*(\omega)=\mathfrak{T}_\eta^*(d\theta)=d\mathfrak{T}_\eta^*(\theta)=d(\theta+\pi^*{\eta})=d\theta=\omega
		%		\]
		%		where the second last equality follows from $\eta$ being closed. 
		%	\end{proof}
	
	The following Proposition is analogous to the fact that the vanishing cycle of $f$ is supported on the critical locus of $f$. 
	\begin{proposition}\label{prop_L}
		Let $X$ be a complex manifold and let $Z$ be a compact irreducible analytic subvariety of $X$. Denote by $\Lambda=T^*_Z X$. If $\eta$ is a $d$-valued closed one-form on $X$, then
		\begin{equation}\label{eq_Phi}
			\Phi_\eta(\Lambda)=n_0\Lambda+\sum_{1\leq i\leq m} n_i \;T^*_{Z_i} X
		\end{equation}
		where $n_0$ is the multiplicity of the zero form in $\eta|_{Z_{\reg}}$ and all $Z_i$ are proper closed subvarieties of $Z$. In particular, if $\eta|_{Z_{\reg}}$ is non-trivial, then $n_0<d$. 
	\end{proposition}
	\begin{proof}%[Proof of Proposition \ref{prop_L}]
		It follows from definition that $\Phi_{\eta}\subset T^*X|_{Z}$. Thus, by Proposition \ref{prop_Lagrangian}, we have the presentation \eqref{eq_Phi} with $n_i\geq 0$ for all $i$. The only remaining statement is that $n_0$ is equal to the multiplicity of the zero form in $\eta|_{Z_{\reg}}$. By \cref{lemma_Z}, it suffices to prove the statement in the case when $Z=X$ and $\Lambda=T^*_X X$. Restricting to a small ball in the unbranching locus of $\eta$, we can also assume that $\eta$ is the union of $d$ single-valued one-forms $\eta_j, 1\leq j\leq d$. In this case, 
		\[
		\Phi_\eta(T^*_X X)=\sum_{1\leq j\leq d}\Phi_{\eta_i}(T^*_X X).
		\]
		Obviously, if $\eta_j=0$, then $\Phi_{\eta_j}(T^*_X X)=T^*_X X$. Otherwise, $\Phi_{\eta_j}(T^*_X X)$ is supported on a proper closed subset of $X$. Thus, $n_0$ is equal to the number of $\eta_j$ which are zero one-forms. 
	\end{proof}
	
	\begin{remark}
		Assume that $x\in Z_{\reg}$ is in the unbranching locus of $\eta$. If near $x$, some branch of $\eta|_{Z_{\reg}}$ has nonempty degenerating locus, then it follows from the definition of $\Phi_\eta$ that $n_i>0$ for some $0\leq i\leq m$. Furthermore, if the degenerating locus has dimension strictly less than $\dim Z$, then $n_i>0$ for some $1\leq i\leq m$. 
	\end{remark}
	
	%It is easy to check that for any $s\in \C^*$, $\Lambda_\eta^\diamond\cap (T^*X\times \{s\})$ is a Lagrangian cycle. Since the limit of Lagrangian cycle is also Lagrangian, $\Phi_\eta \Lambda$ is a Lagrangian cycle. 
	
	%\textcolor{red}{The proof of being Lagrangian is not correct. We may need to postpone this part to Proposition \ref{prop_L}.}
	
	We have proved that $\Phi_\eta$ maps conic Lagrangian cycles to conic Lagrangian cycles. Hence it induces a group homomorphism $\Phi_\eta: L(X)\to L(X)$, which we call \emph{the vanishing cycle functor} of $\eta$. The following proposition justifies this name. 
	%\end{definition}
	
	\begin{proposition}[Massey]
		If $f$ is a holomorphic function on a complex manifold $X$, then $\Phi_{df}$ is the total vanishing cycle functor. In other words, given a conic Lagrangian cycle $\Lambda$ in $T^*X$,
		\begin{equation}\label{eq_1}
			\Phi_{df}(\Lambda)=\sum_{t\in \C}\Phi_{f-t}(\Lambda)
		\end{equation}
		where $\Phi_{f-t}$ is the standard vanishing cycle functor of the holomorphic function $f-t$. Note that the sum on the right-hand side is a locally finite sum, because restricting to a small ball in $X$, there are only finitely many $t$ such that $\Phi_{f-t}(\Lambda)\neq 0$. 
	\end{proposition}
	\begin{proof}
		In this case, $\Gamma_\eta$ is a section of $T^*X$. 
		Then $\Phi_\eta\Lambda$ is the deformation to the normal cone (Remark \ref{remark_def}), the equation \eqref{eq_1} is equivalent to \cite[Theorem 2.10]{Massey}. 
	\end{proof}
	
	\subsection{Some properties of $\Phi_\eta$}
	%By definition $\Phi_\eta$ sends effective cycles in $L(X)$ to effective cycles. We list some further properties of $\Phi_\eta$. 
	
	\begin{proposition}\label{prop_d}
		Suppose $\eta$ is a $d$-valued closed one-form. For any $\Lambda\in L(X)$, we have
		\[
		d\, \Lambda\cdot T_X^* X=\Phi_\eta(\Lambda)\cdot T^*_X X
		\]
		where $\cdot$ denotes the intersection number in $T^*X$. 
	\end{proposition}
	\begin{proof}
		Recall that $\Phi_\eta\Lambda$ is defined to be the restriction of $\overline{\Lambda_\eta^\diamond}$ to $T^*X\times \{\infty\}$. Since the intersection number is invariant under rational equivalence, 
		\[
		\Phi_\eta(\Lambda)\cdot T^*_X X=\big(\overline{\Lambda_\eta^\diamond}\cap (T^*X\times \{0\})\big)\cdot T^*_X X.
		\]
		By definition, we have $\overline{\Lambda_\eta^\diamond}\cap (T^*X\times \{0\})=d \Lambda$ as $n$-cycles on $T^*X$. Therefore, the desired equality follows. 
		%Without loss of generality, we can assume that $\Lambda$ is irreducible. Since $\eta$ is a $d$-valued one-form, the scheme theoretic intersection $\Lambda^\diamond\cap (T^* X\times \{0\})$ is rationally equivalent to $d \Lambda$. The desired assertion follows from the fact that intersection number is invariant under rational equivalence. \textcolor{blue}{To be corrected. }
	\end{proof}

	\begin{corollary}\label{cor_subvariety}
		Let $X$ be a complex manifold and $Z$ be a compact irreducible subvariety of $X$. If there exists a closed multivalued one-form $\eta$ on $X$ whose restriction to $Z$ is non-trivial, then there exist proper closed subvarieties $Z_i$ of $Z$ and $\lambda_i \in \Q_{\geq 0}$ such that
		\[
		T^*_Z X\cdot T^*_X X=\sum_{1\leq i\leq m} \lambda_i \;T^*_{Z_i} X\cdot T^*_X X.
		\]
	\end{corollary}
	\begin{proof}
		By Propositions \ref{prop_L} and \ref{prop_d}, we have
		\[
		d T^*_Z X\cdot T_X^* X=\Phi_\eta(\Lambda)\cdot T^*_X X=\Big(n_0T^*_Z X+\sum_{1\leq i\leq m} n_i \;T^*_{Z_i} X\Big)\cdot T^*_X X.
		\]
		By the assumption that the restriction of $\eta$ to $Z_{\reg}$ is non-trivial, $n_0<d$, and we have
		\[
		T^*_Z\cdot T_X^* X=\sum_{1\leq i\leq m} \frac{n_i}{d-n_0} \;T^*_{Z_i} X\cdot T^*_X X. \qedhere
		\]
	\end{proof}
	
	The classical vanishing cycle functor is defined for constructible complexes. So we end the section with the following question.
	\begin{question}
		Let $\eta$ be a multivalued closed one-form on a complex manifold $X$. Is there a natural lifting of $\Phi_\eta$ to a functor of constructible complexes $\Phi_\eta: D^b_c(X, \C)\to D^b_c(X, \C)$?
	\end{question}
	
	\subsection{Proof of Theorem \ref{thm_main} in the case when $\mathrm{char}(K)>0$}
	\begin{proposition}\label{prop:positive}
		Let $X$ be a smooth projective variety and let $\rho:\pi_1(X)\to \GL(r,K)$ be a linear representation where $K$ is a field of characteristic $p>0$. If $\rho$ is large, then there exists a closed multivalued one-form $\eta$ on $X$ such that for any positive-dimensional closed subvariety $Z$ of $X$, the restriction of $\eta$ to $Z|_{\reg}$ is non-trivial.
	\end{proposition}
	\begin{proof}
		By \cite[Theorem 2.7 \& Corollary 2.10]{DY24}, there exists a semisimple representation $\tau:\pi_1(X)\to \GL(N,L)$ where $L$ is a non-archimedean local field of characteristic $p$ such that the Katzarkov-Eyssidieux reduction map $s_\tau$ is the Shafarevich morphism of $\rho$.  Since $\rho$ is large, its Shafarevich morphism is the identity map. Hence $s_\tau$ is the identity map. Let $\eta$ be the associated closed multivalued one-form in Proposition \ref{thm:KE}. By the equivalence of \cref{item:KZ2} and \cref{item:KZ4} in Proposition \ref{thm:KE}, for any positive dimensional subvariety $Z$, the restriction $\eta|_{Z_{\reg}}$ is non-trivial. The proposition is proved. 
	\end{proof}
	
	\begin{proof}[Proof of Theorem \ref{thm_main} assuming $\mathrm{char}(K)>0$]
		In this case the axioms of Proposition \ref{prop:positive} are satisfied, and hence we have a closed multivalued one-form $\eta$ whose restriction to any positive dimensional subvariety is non-trivial. Iterating Corollary \ref{cor_subvariety}, we can express $T^*_Z X\cdot T^*_X X$ as a sum of finitely many $\lambda_i T^*_{Z_i} X\cdot T^*_X X$, where $\lambda_i>0$ and $Z_i$ is a point. If $Z_i$ is a point, then $T^*_{Z_i} X\cdot T^*_X X=1$. Thus, $T^*_Z X\cdot T^*_X X\geq 0$.
	\end{proof}
	
	\section{Period maps of $\C$-VHS and positivity}
	\subsection{Positivity from the period maps}
	In this subsection, we prove the following generalization of \cite[Theorem 1.9]{AW}. 
	\begin{proposition}\label{prop_geq0}
		Let $X$ be a projective manifold with a representation $\sigma: \pi_1(X)\to \GL(N, \C)$ such that the associated local system $L_\sigma$ underlies a $\C$-VHS. Assume that $Z$ is an irreducible subvariety of $X$ such that 
		\begin{enumerate}
			\item the pullback of ${\sigma}$ to $\pi_1(Z_{\norm})$ is large;
			\item $\Gamma\coloneqq \sigma(\mathrm{Im}[\pi_1(Z_{\norm})\to \pi_1(X)])$ is a discrete subgroup of $\GL(N, \C)$. 
		\end{enumerate}
		Then the intersection number $T_Z^*X \cdot T_X^* X\geq 0$. Equivalently, $(-1)^{\dim Z}\chi(Eu_Z)\geq 0$. 
	\end{proposition}
	When $Z$ happens to be normal, the proposition follows immediately from the arguments of \cite{AW}. However, some technical arguments are required to deal with the non-normal case. 
	
	First, we need to reduce to the case when $\Gamma$ is torsion free. 
	
	\begin{lemma}
		Let $X$ be a projective manifold, and $\pi: \widetilde{X}\to X$ be a finite covering map. For an irreducible subvariety $Z$ of $X$, let $\widetilde{Z}$ be a connected component of $\pi^{-1}(Z)$. Then,
		\[
		T_{\widetilde{Z}}^*\widetilde{X}\cdot T^*_{\widetilde{X}}\widetilde{X}=d \; T_Z^* X \cdot T^*_X X
		\] 
		where $d$ is the degree of the covering map $\pi|_{\widetilde{Z}}: \widetilde{Z}\to Z$. 
	\end{lemma}
	\begin{proof}
		Since $\pi|_{\widetilde{Z}}: \widetilde{Z}\to Z$ is a covering map of degree $d$, the pushforward of $Eu_{\widetilde{Z}}$ under $\pi|_{\widetilde{Z}}$ is equal to $d \cdot Eu_Z$. Hence 
		\[
		\chi\big(Eu_{\widetilde{Z}}\big)=\chi\big(\pi|_{\widetilde{Z}*}(Eu_{\widetilde{Z}})\big)=d\cdot \chi(Eu_Z).
		\]
		By \eqref{eq_Eu}, the desired equation follows. 
	\end{proof}
	
	By \cite[Lemma 8]{Sel}, $\sigma(\pi_1(X))$ has a torsion-free finite index subgroup. Replacing $X$ by the induced finite cover, without loss of generality, we can assume that $\sigma(\pi_1(X))$ is torsion free.   	
	Then as a subgroup, $\Gamma$ is also torsion free. 
	
	By the argument of \cite{AW}, we know that there is a finite period map $Z_{\norm}\to \Gamma\backslash D$. Since $\Gamma\backslash D$ has non-positive holomorphic bisectional curvature in the horizontal directions, we can deduce the fact that $(-1)^{\dim Z}\chi(Eu_{Z_{\norm}})\geq 0$. However, this is weaker than the inequality $(-1)^{\dim Z}\chi(Eu_{Z})\geq 0$. To achieve the latter inequality, we introduce a new constructible function $\delta$ on $Z_{\norm}$ as follows. 
	
	\begin{definition}\label{defn_delta}
		Let $Z$ be an algebraic variety and let $p: Z_{\norm}\to Z$ be the normalization map. Given any point $x\in Z_{\norm}$, we choose a small open neighborhood $U_x\subset Z_{\norm}$ of $x$. The value of $\delta$ at $x$ is defined to be the value of $(-1)^{\dim Z}Eu_{p(U_x)}$ at $p(x)$. 
	\end{definition}
	The function $\delta$ is constructible with respect to a stratification of the map $p: Z_{\norm}\to Z$ (see \cite[Chapter 1, 1.6 and 1.7]{GM}).
	
	\begin{example}
		Suppose that $Z$ is a projective curve with two singular points: a cusp $P$ and a node $Q$. Then the above defined function $\delta$ has value -1 on every point of $Z_{\norm}$ except at the preimage $P'$ of $P$, where the value is -2. In this case, $\delta$ is not CC-effective on $Z_{\norm}$, because
		\[
		\delta=(-1)^{\dim Z}1_{Z_{\norm}}-1_{P'},
		\]
		where $1_{Z_{\norm}}=Eu_{Z_{\norm}}$, since $Z_{\norm}$ is smooth. Such $Z$ could appear in Proposition \ref{prop_geq0}, and in such case, we will see that the image of $Z_{\norm}$ in the period domain $\Gamma\backslash D$ must at least have a cusp singularity at the image of $P'$. So the pushforward of $\delta$ in $\Gamma\backslash D$ will be CC-effective. 
	\end{example}
	
	\begin{lemma}\label{lemma_chi}
		Let $p: Z_{\norm}\to Z$ and the constructible function $\delta$ be as in Definition \ref{defn_delta}. If $U\subset Z_{\norm}$  is open, then $(p|_{U})_*(\delta)=(-1)^{\dim Z}Eu_{p(U)}$. 
		In particular, $p_*(\delta)=(-1)^{\dim Z}Eu_Z$ and $\chi(\delta)=(-1)^{\dim Z}\chi(Eu_Z)$. 
	\end{lemma}
	\begin{proof}
		The lemma follows from the definition of $\delta$ and the fact that for a reducible analytic variety, its Euler obstruction function equals to the sum of the Euler obstruction of every irreducible components. 
	\end{proof}

	Let $\pi: \widetilde{X}\to X$ be the covering map induced by $\ker(\sigma)$. %Let $\widetilde{Z}$ be a connected component of $\pi^{-1}(Z)$. 
	Let $\pi_{Z_{\norm}}: \widetilde{Z}_{\norm}\to Z_{\norm}$ be the covering map induced by $\ker[\pi_1(Z_{\norm})\to \pi_1(X)\xrightarrow{\sigma}\GL(N, \C)]$. Then the map $Z_{\norm}\to X$ lifts to a map $\widetilde{Z}_{\norm}\to \widetilde{X}$. Let $\widetilde{Z}$ be the image of $\widetilde{Z}_{\norm}$ in $\widetilde{X}$. Then we have the following commutative diagram:
	%Since $\pi|_{\widetilde{Z}}: \widetilde{Z}\to Z$ is a covering map, we can find a lifting map $\tilde{p}: \widetilde{Z}_{\norm}^{\univ}\to \widetilde{Z}$ such that the following diagram commutes. %\textcolor{blue}{We should replace the universal cover of $Z_{\norm}$ by the smaller one induced by the kernel of $\pi_1(Z_{\norm})\to \pi_1(X)$ to be consistent with the arguments in next section. }
	\[
	\xymatrix{
		\widetilde{Z}_{\norm}\ar[d]^{\pi_{Z_{\norm}}}\ar[r]^{\tilde{p}}&\widetilde{Z}\ar[d]^{\pi_{Z}}\ar[r]^{\tilde\iota}&\widetilde{X}\ar[d]^{\pi}\\
		Z_{\norm}\ar[r]^{p}& Z\ar[r]^{\iota}&  X.
	}
	\]
	It follows from definition that $\widetilde{Z}_{\norm}$ is a connected component of the fiber product $Z_{\norm}\times_{X} \widetilde{X}$. Since $\iota\circ p: Z_{\norm}\to X$ is a finite morphism, so are $\tilde\iota\circ \tilde{p}: \widetilde{Z}_{\norm}\to \widetilde{X}$ and $\tilde{p}: \widetilde{Z}_{\norm}\to \widetilde{Z}$ (of analytic varieties). 
	\begin{remark}
		The map $\pi_Z$ is a covering map over the Zariski open subset of $Z$ where it is locally irreducible. For example, if $Z$ is curve, everywhere smooth except at one nodal singular point, then $\widetilde{Z}$ is a covering space of either $Z$ or $Z_{\norm}$ depending on whether the two compositions $\pi_1(Z_{\norm})\to \pi_1(X)\xrightarrow{\sigma} \GL(N, \C)$ and $\pi_1(Z)\to \pi_1(X)\xrightarrow{\sigma}  \GL(N, \C)$ have the same image or not. 
	\end{remark}
	\begin{lemma}\label{lemma_tildep}
		Let $\tilde{\delta}=\pi_{Z_{\norm}}^*(\delta)$. Then $\tilde{p}_*(\tilde{\delta})$ is CC-effective. 
	\end{lemma}
	\begin{proof}
		Since $\iota$ is a closed embedding, it suffices to show that $(\tilde{\iota}\circ \tilde{p})_*(\tilde{\delta})$ is CC-effective. By definition, near the image of $\tilde{\iota}\circ \tilde{p}$, $(\iota\circ \tilde{p})_*(\tilde{\delta})$ is equal to $\pi^*((\iota\circ p)_*\delta)$. Since $\pi$ is a covering map, it suffices to show that $(\iota\circ p)_*\delta$ is CC-effective. By Lemma \ref{lemma_chi}, $p_*\delta$ is CC-effective. Since $\iota$ is a closed embedding, $(\iota\circ p)_*\delta$ is also CC-effective.
	\end{proof}
	
	%\begin{lemma}
	%Given any point $x\in \widetilde{Z}_{\norm}$, denote the analytic germ of $\widetilde{Z}_{\norm}$ at $x$ by $\widetilde{Z}_{\norm, x}$. Then $\tilde{p}(\widetilde{Z}_{\norm, x})$ is an irreducible component of the analytic germ $\widetilde{Z}_{\tilde{p}(x)}$. Moreover, $\tilde{p}(\widetilde{Z}_{\norm, x})\cong \pi_Z\circ \tilde{p}(\widetilde{Z}_{\norm, x})$ and $\pi_Z\circ \tilde{p}(\widetilde{Z}_{\norm, x})$ is an irreducible component of the analytic germ $Z_{\pi_Z\circ\tilde{p}(x)}$.
	%\end{lemma}
	%\begin{proof}
	%Denote the normalization of $\widetilde{Z}$ by $(\widetilde{Z})_{\norm}$. Then by the universal property of normalization, the maps $\tilde{p}: \widetilde{Z}_{\norm}\to \widetilde{Z}$ and $(\widetilde{Z})_{\norm}\to Z$ factor as 
	%\[
	%\widetilde{Z}_{\norm}\xrightarrow{\tilde{p}_1} (\widetilde{Z})_{\norm}\xrightarrow{\tilde{p}_2}  \widetilde{Z}\quad \text{and}\quad (\widetilde{Z})_{\norm}\to {Z}_{\norm}\to Z
	%\]
	%respectively. 
	%Since $\pi_{Z_{\norm}}:\widetilde{Z}_{\norm}\to {Z}_{\norm}$ is a covering map, the maps $\widetilde{Z}_{\norm}\to (\widetilde{Z})_{\norm}$ and $(\widetilde{Z})_{\norm}\to {Z}_{\norm}$ must also be covering maps. Thus, $\tilde{p}_1(\widetilde{Z}_{\norm, x})= (\widetilde{Z})_{\norm, \tilde{p}_1(x)}$. Since $(\widetilde{Z})_{\norm}\to \widetilde{Z}$ is the normalization map, we can deduce the first statement of the lemma. 
	%\end{proof}

	By the definition of $\widetilde{X}$, the pullback of $L_\sigma$ to $\widetilde{X}$ becomes a trivial local system. Thus, the $\bC$-VHS structure on $L_\sigma$ induces a period map 
	\[
	\widetilde{\phi}: \widetilde{X}\to D.
	\]

	%Fixing any lift of the composition $Z_{\norm}\to Z\to X$ to the universal cover $\widetilde{Z}_{\norm}^{univ}$ of $Z_{\norm}$, 
	Denote the pullback of $\widetilde{\phi}$ to $\widetilde{Z}$ by $\widetilde{\phi}_Z: \widetilde{Z}\to D$. 
	Now, the deck transformation group of the normal covering map $\widetilde{Z}_{\norm}\to {Z}_{\norm}$ is equal to $\Gamma$, and $\Gamma$ acts on $D$ via the monodromy action of the $\bC$-VHS $L_{\sigma}$. Thus, $\Gamma$ acts equivariantly on the composition $\tilde{\phi}\circ\tilde{p}: \widetilde{Z}_{\norm}\to  D$, and taking quotient by the $\Gamma$-action, we have the commutative diagram
	\begin{equation}\label{eq_diagram}
		\xymatrix{
			\widetilde{Z}_{\norm}\ar[d]\ar[r]^{\tilde{p}} &\tilde{Z}\ar[r]^{\tilde{\phi}_Z} &D\ar[d]^{\pi_D}\\
			Z_{\norm}\ar[rr]^{\phi}&& \Gamma\backslash D.
		}
	\end{equation}
	Since $\Gamma$ is torsion free and discrete, we know that $\Gamma\backslash D$ is a complex manifold. 
	
	Since the pullback of $\sigma$ to $\pi_1(Z_{\norm})$ is large, by \cite[Proof of Theorem 1.9]{AW}, we have the following.
	\begin{lemma}\label{lemma_finite}
		The period map $\phi: Z_{\norm}\to \Gamma\backslash D$ is a finite morphism.  \qed
	\end{lemma}
	\begin{proposition}\label{prop_delta}
		Let $\delta$ be the constructible function as in Definition \ref{defn_delta}. Then the pushforward $\phi_{*}(\delta)$ is CC-effective. 
	\end{proposition}
	\begin{proof}
		The proof will be based on the commutative diagram \eqref{eq_diagram}. 
		
		Let $\tilde{\delta}$ be the pullback of $\delta$ to the covering space $\widetilde{Z}_{\norm}$. Notice that ignoring $\widetilde{Z}$, \eqref{eq_diagram} is 
		%replace $D$ and $\Gamma\backslash D$ by $\im(\tilde{\phi}_Z\circ \tilde{p})$ and $\im(\phi)$ in \eqref{eq_diagram}, then 
		a Cartesian square. Thus, it suffices to show that $(\tilde{\phi}_Z\circ \tilde{p})_*(\tilde{\delta})$ is CC-effective. By Lemma \ref{lemma_tildep}, $\tilde{p}_*(\tilde{\delta})$ is CC-effective. Since $\tilde{\phi}_Z$ is a finite morphism, by Proposition \ref{prop_AMSS},  $(\tilde{\phi}_Z\circ \tilde{p})_*(\tilde{\delta})$ is CC-effective. 
	\end{proof}
	
	\begin{remark}\label{remark_horizontal}
		%In the tangent bundle $TD$ there is a subbundle $T^hD$ called the horizontal subbundle (see e.g., \cite{CMP, P}). 
		The horizontal subbundle $T^{-1, 1}D$ defined in \cref{sec:period} is preserved by the $\Gamma$-action, and hence descends to a subbundle $T^{-1, 1}(\Gamma\backslash D)$. A subvariety $Y$ of $\Gamma\backslash D$ is called horizontal if the image of $TY_{\reg}$ in $T(\Gamma\backslash D)|_{Y_{\reg}}$ is contained in $T^{-1, 1}(\Gamma\backslash D)|_{Y_{\reg}}$. The image of a period map is always horizontal. It is proved in \cite[Corollary 5.4]{AW} that the subvariety of a horizontal variety is also horizontal. 
	\end{remark}
	The following proposition was proved using the curvature property of the period domain. 
	\begin{proposition}\cite[Proposition 6.1]{AW}\label{prop_AW}
		Let $Y$ be a horizontal compact irreducible analytic subvariety of $\Gamma\backslash D$. Then 
		\[
		T^*_Y \Gamma\backslash D\cdot T^*_{\Gamma\backslash D} \Gamma\backslash D\geq 0
		\]
		as intersection number in $T^*\Gamma\backslash D$. Equivalently, $(-1)^{\dim Y}\chi(Eu_Y)\geq 0$. 
	\end{proposition}

	\begin{proof}[Proof of Proposition \ref{prop_geq0}]
		By Lemma \ref{lemma_chi}, 
		\begin{equation}\label{eq_01}
			(-1)^{\dim Z}\chi(Eu_Z)=\chi(\delta).
		\end{equation}
		Since $\phi: Z_{\norm}\to \Gamma\backslash D$ is a finite morphism, 
		\begin{equation}\label{eq_02}
			\chi(\delta)=\chi(\phi_*(\delta)).
		\end{equation}
		Moreover, the support of $\phi_*(\delta)$ is compact. 
		By Proposition \ref{prop_delta}, we know that 
		\begin{equation}\label{eq_Yi}
			\phi_*(\delta)=\sum_{1\leq i\leq m}(-1)^{\dim Y_i}n_i Eu_{Y_i}
		\end{equation}
		for irreducible compact subvarieties $Y_i$ of $\Gamma\backslash D$ and $n_i>0$. By Remark \ref{remark_horizontal}, $\phi(Z_{\norm})$ is horizontal. Since all $Y_i$ are contained in $\phi(Z_{\norm})$, they are also horizontal. Thus, by Proposition~\ref{prop_AW}, $(-1)^{\dim Y_i}\chi(Eu_{Y_i})\geq 0$ for all $i$. Therefore, by \eqref{eq_Yi}, $\chi(\phi_*(\delta))\geq 0$. By \eqref{eq_01} and \eqref{eq_02}, $(-1)^{\dim Z}\chi(Eu_Z)\geq 0$.
	\end{proof}
	
	\subsection{Proof of Theorem \ref{thm_main} in the case when $K=\C$ and $\rho$ is semisimple}\label{sec:semisimple}
	Let $\rho: \pi_1(X)\to \GL(r, \C)$ be a large and semisimple representation. 
	The proof of the reductive Shafarevich conjecture as in \cite{DY23} gives us the desired multivalued one-forms and $\C$-VHS as in the following proposition. 
	%Using results from the reductive Shafarevich conjecture proven in \cite{DY23}, we will show that there exist a multivalued one-form $\eta$ on $X$ and  a $\bC$-VHS $\sigma$ on $X$ such that,  for  any positive-dimensional subvariety $Z$ of $X$, either the restriction of $\eta$ to  $Z$ is non-trivial, or $p^*\sigma$ is a large representation with discrete monodromy, where $p:Z_{\rm norm}\to Z$ is the normalization. 
	
	\begin{proposition}\label{prop:Shafarevich}
		Let $X$ be a projective manifold and let $\rho:\pi_1(X)\to \GL(r,\bC)$ be a large and semisimple representation.   Then there exists a  semisimple representation $\sigma:\pi_1(X)\to \GL(N,\bC)$ on $X$ underlying a $\bC$-VHS, and a closed multivalued one-form $\eta$ on $X$ such that for any irreducible subvariety $Z$ of $X$, one of the following two statements holds.
		\begin{enumerate}
			\item The restriction of $\eta$  to   $Z$ is non-trivial.
			\item The pullback of $\sigma$ to the normalization $Z_{\rm norm}$ of $Z$ is large and has discrete monodromy.
		\end{enumerate}
	\end{proposition}
	\begin{proof} 
		Let $s_{{\rm fac},r}:X\to S_{{\rm fac},r}$ be the fibration defined in \cref{def:reduction ac}.  By  \cite[Proposition 3.13]{DY23},  	  there exist a semisimple representation $\sigma:\pi_1(X)\to \GL(N,\bC)$  underlying  a $\bC$-VHS,  such that for any  closed subvariety $Z$ with $s_{{\rm fac},r} (Z)$  being a point, the following properties hold:
		\begin{enumerate} 
			\item  \label{item:directsume}  let $p:Z_{\rm norm}\to Z$ be the normalization map.   Then the image of $p^*\sigma: \pi_1(Z_{\norm})\to \GL(N,\bC)$ is a discrete subgroup. 
			\item \label{item:max}For each semisimple representation $\tau:\pi_1(X)\to {\rm GL}(r,\bC)$,   $p^*\tau$ is conjugate to a direct factor of  $p^*\sigma:\pi_1(Z_{\norm})\to \GL(N,\bC)$.  
		\end{enumerate}  Since $\rho$ is large, $p^*\rho:\pi_1(Z_{\norm})\to \GL(N,\bC)$ is also large.   Item \ref{item:max} implies that $p^*\sigma:\pi_1(Z_{\norm})\to \GL(N,\bC)$ is large.

		By \cite[Theorem 1.28]{DY23}, there exist  semisimple representations $\{\rho_i:\pi_1(X)\to \GL(r,K_i)\}_{i=1,\ldots,k}$ where each $K_i$ is a non-archimedean local field of characteristic zero such that   the Stein factorization of $(s_{\rho_1},\ldots,s_{\rho_k}):X\to S_{\rho_1}\times \cdots\times S_{\rho_k}$ coincides with $s_{{\rm fac},r}:X\to S_{{\rm fac},r}$. Recall in  \cref{thm:KE}, each $\rho_i$ gives rise to a  multivalued one-form $\eta_i$ on $X$ such that the properties in  \cref{thm:KE} are equivalent. Let $\eta$ be the union of $\eta_1,\ldots,\eta_k$.
		
		Suppose that $Z$ is a closed subvariety of $X$ such that $\eta|_{Z_{\reg}}$ is trivial. Then we need to show that $s_{{\rm fac},r}(Z)$ is not a point. In fact, since $\eta|_{Z_{\reg}}$ is trivial, each $\eta_i|_{Z_{\reg}}$ is trivial. By \cref{thm:KE}, $s_{\rho_i}(Z)$ is a point for each $i$. By the property of the simultaneous Stein factorization, $s_{{\rm fac},r}(Z)$ is also a point. So the above property \eqref{item:directsume} holds, that is, the pullback of $\sigma$ to $Z_{\norm}$ is large and has discrete monodromy.  The proposition is proved.
		%\begin{claim}
		%For any  closed subvariety $Z\subset X$ such that  $s_{{\rm fac},r}(Z)$ is not a point, the restriction of $\eta$ to $Z$ is non-trivial. 
		%\end{claim}   
		%\begin{proof}
		%Assume by contradiction that $\eta|_{Z}$ is trivial. Then  $\eta_i|_Z$ is trivial for each $i$. By  Proposition \ref{thm:KE}, we have $s_{\rho_i}(Z)$ is a point for each $i$. Therefore, $s_{{\rm fac},r}(Z)$ is a point by the property of $s_{{\rm fac},r}$, leading the contradiction. This proves that $\eta|_{Z}$ is non-trivial. The closedness of $\eta$ follows from Proposition \ref{thm:KE} directly. 
		%\end{proof} 
		%The proposition is proved. 
	\end{proof}
	
	Therefore, given any conic Lagrangian cycle $\Lambda\subset T^*X$,   we can apply Proposition \ref{prop:Shafarevich} and Corollary \ref{cor_subvariety} to   iterate the vanishing cycle functor $\Phi_\eta$ to achieve the following. 
	
	\begin{corollary}\label{cor_2properties}
		Let $X$ be a smooth projective variety and let $\rho:\pi_1(X)\to \GL(r,\bC)$ be a semisimple and large representation. Let $\sigma:\pi_1(X)\to \GL(N,\bC)$ be the representation constructed in Proposition \ref{prop:Shafarevich}, and let $L_{\sigma}$ be the associated local system. Then, given any conic Lagrangian cycle $T^*_Z X$, there exist subvarieties $Z_i$ and rational numbers $\mu_i> 0$ with $1\leq i\leq m$ satisfying the following properties. 
		\begin{enumerate}
			\item As intersection numbers in $T^*X$, we have
			\[
			T^*_Z X\cdot T^*_X X=\sum_{1\leq i\leq m} \mu_i T^*_{Z_i} X\cdot T^*_X X.
			\]
			\item For each $i$, the pullback of $L_{\sigma}$ to the normalization $Z_{i, \rm norm}$ of $Z_i$ is large and has discrete monodromy. 
		\end{enumerate}
	\end{corollary}
	\begin{remark}
		If $\eta=df$ for some holomorphic function $f$ on $X$, then the support of $\Phi_\eta(T^*_{Z}X)$ in $X$ is contained in the critical locus of $f|_{Z}$  in the stratified sense. In particular, $f$ has constant value on each irreducible component of the support of $\Phi_\eta(T^*_{Z}X)$. Consequently, $\Phi_\eta\circ \Phi_\eta(T^*_{Z}X)=\Phi_\eta(T^*_{Z}X)$. 
		However, when $\eta$ is a closed multivalued one-form on $X$, the above identity may not hold even up to a scalar. In fact,  over a small ball in the unbranching locus $X^\circ$, $\eta$ can be considered as a union of closed one-forms $\eta_1, \ldots, \eta_l$ (possibly with multiplicity), and locally we can assume $\eta_i=df_i$. Then locally, $\Phi_\eta=\sum_{1\leq i\leq l}\Phi_{df_i}$. Even though $\Phi_{df_i}\circ \Phi_{df_i}(T^*_{Z}X)=\Phi_{df_i}(T^*_{Z}X)$, we have no control about the terms $\Phi_{df_i}\circ \Phi_{df_j}(T^*_{Z}X)$. 
	\end{remark}
	
	\begin{proof}[Proof of Theorem \ref{thm_main} in the case when $K=\C$ and $\rho$ is semisimple]
		By Corollary \ref{cor_2properties} (1), it is sufficient to show that $T^*_{Z_i} X\cdot T^*_X X\geq 0$, where $Z_i$ satisfies the conditions in Corollary \ref{cor_2properties} (2). By Proposition \ref{prop_geq0}, the above inequality holds. 
	\end{proof}

	\section{Mixed period maps of $\bR$-VMHS and positivity}\label{sec:mixed}
	\subsection{Positivity from the mixed period maps}
	In this subsection, we prove the following analogue of \cref{prop_geq0} for mixed period maps.
	\begin{proposition}\label{prop_linear}
		Let $X$ be a projective manifold with two representations $\{\sigma_i: \pi_1(X)\to \GL(N_i, \C)\}_{i=1, 2}$ such that
		\begin{enumerate}
			\item $\sigma_1$ is semisimple and $L_{\sigma_1}$ underlies a $\C$-VHS;
			\item $L_{\sigma_2}$ underlies a $\bR$-VMHS of weight $-1,0$. 
		\end{enumerate}
		Let $Z$ be an irreducible subvariety of $X$, and let $p:Z_{\norm}\to Z$ be its normalization. Let $\Gamma_1:=\sigma_1({\rm Im}[\pi_1(Z_{\rm norm}\to \pi_1(X)])$ %$\Gamma_2:=\sigma_2({\rm Im}[\pi_1(Z_{\rm norm})\to \pi_1(X)])$ 
		and $\Gamma_2^{ss}:=\sigma^{ss}_2({\rm Im}[\pi_1(Z_{\rm norm})\to \pi_1(X)])$, where $\sigma^{ss}_2: \pi_1(X)\to \GL(N_2, \C)$ is the semisimplification of $\sigma_2$, which underlies a $\bC$-VHS $L_{\sigma_2^{ss}}$.  Assume that 
		\begin{enumerate}
			\item[(i)]     the product of the period map of $p^*\sigma_1$ and the mixed period map of $p^*\sigma_2$, denote by $\widetilde{Z}_{\rm norm}^{\rm univ}\to \cD_1\times \cM$,  has   discrete fibers. Here $\widetilde{Z}_{\rm norm}^{\rm univ}$ is the universal cover of $Z_{\rm norm}$, $\cD_1$ is the period domain of $L_{\sigma_1}$, and $\cM$ is the mixed period domain of $L_{\sigma_2}$.  
			\item[(ii)] The monodromy group  $\Gamma_1$  (resp. $\Gamma_2^{ss}$) acts on $\cD_1$ (resp. $\cD_2$) respectively.  Here $\cD_2$ is the period domain of  the $\bC$-VHS $L_{\sigma_2^{ss}}$.
		\end{enumerate}
		Then, the intersection number $T^*_Z X\cdot T^*_X X\geq 0$. Equivalently, $(-1)^{\dim Z}\chi(Eu_Z)\geq 0$. 
	\end{proposition}
	
	Using the same finite covering trick as in the beginning of the previous section, we can reduce to the case when $\Gamma_1$ and $\Gamma_2^{ss}$ are torsion free. We will work under this assumption for the remaining of this subsection. 
	
	Since the image of $\sigma_2: \pi_1(X)\to \GL(N_2, \bR)$ may not be discrete, we will need a technical proposition to deduce desired positivity. Before stating the proposition, we need a lemma. 
	
	\begin{lemma}\label{lemma_ef}
		Let $0\to N\to M\to M'\to 0$ be a short exact sequence of holomorphic vector bundles on a complex manifold $X$. Assume that $M'$ is trivial of rank $r$. Consider $N$ as a complex submanifold of $M$, and denote by $\iota:N\to M$ the inclusion map. Let $C$ be an equi-dimensional subvariety of $M$, and $[C]\in H^{BM}_{2\dim C}(C, \Z)$ be the fundamental class of $C$.
		Then $\iota^*[C]\in H^{BM}_{2\dim C-2r}(C\cap N, \Z)$ can be represented by an effective analytic $(\dim C-r)$-cycle in $C\cap N$, where the pullback class $\iota^*[C]$ is defined as the Borel-Moore homology analogue of \cite[Chapter 8]{Fulton}. 
	\end{lemma}
	\begin{proof}
		Since $M'$ is a trivial bundle, the normal bundle of $N$ in $M$ is also trivial. Using the deformation to normal cone \cite[Chapter 5]{Fulton}, we can replace $M$ by a trivial vector bundle over $N$ and $C$ a cone of the new vector bundle $M$. Now, we can take a general global section $\Gamma$ of $M$, and the image of $\Gamma\cap C$ in $N$ is contained in $N\cap C$ and it represents the class $\iota^*[C]$. 
	\end{proof}

	\begin{proposition}\label{prop_E}
		Let $\pi_B: E\to B$ be a holomorphic vector bundle over a complex ball $B$. Suppose that there is a $\Z^k$-action on $E$ by fiberwise translations, i.e., there are sections $s_1, \ldots, s_k$ of the vector bundle, and $(m_1, \ldots, m_k)\in \Z^k$ acts on $E$ by fiberwise addition with $m_1s_1+\cdots +m_ks_k$. 
		Let $\pi_X: \widetilde{X}\to X$ be a $\Z^k$-covering map of complex manifolds. Denote the support of the constructible function $\gamma$ by $\Supp(\gamma)$. Suppose we have a commutative diagram 
		\begin{equation}\label{eq_diag2}
			\xymatrix{
				\widetilde{X}\ar[d]^{\pi_X}\ar[r]^{\tilde{f}}&E\ar[d]^{\pi_B}\\
				X\ar[r]^{f}& B
			}
		\end{equation}
		where $\tilde{f}$ is a $\Z^k$-equivariant holomorphic map. 
		%locally finite morphism, i.e., for any $y\in \widetilde{Y}$, there exists open neighborhood $U$ of $y$ and $V$ of $f(y)$ such that $f: U\to V$ is a finite morphism. 
		%Moreover, we assume that $Y$ can be embedded into some smooth manifold $X$ and the map $f:Y\to B$ can be extend to a map $f_X: X\to B$. 
		If $\gamma$ is a constructible function on $X$ satisfying
		\begin{enumerate}
			\item  for any $e\in E$, $\pi_X^{-1}\Supp(\gamma)\cap \tilde{f}^{-1}(e)$
			is a discrete set,
			\item $f|_{\Supp(\gamma)}: \Supp(\gamma)\to B$ is proper, 
			\item $\gamma$ is CC-effective,
		\end{enumerate}
		then $f_*(\gamma)$ is also CC-effective.%\footnote{The earlier assumption ensures that $\tilde{f}_*\pi_X^*(\lambda)$ is a well-defined constructible function.} 
	\end{proposition}
	\begin{proof}
		%First, notice that property (2) implies that $\tilde{f}_*\pi_X^*(\lambda)$ is a constructible function.
		Since being CC-effective is a local property, by possibly shrinking $B$, we can assume that $\gamma$ is constructible with respect to a finite Whitney stratification of $X$. Furthermore, without loss of generality, we can assume that $\gamma=(-1)^{\dim Z}Eu_{Z}$ for some irreducible closed analytic subvariety $Z$ of $X$. Then $CC(\gamma)=T^*_Z X$, which we denote by $\Lambda$. We will construct some an effective conic Lagrangian cycle on $B$ and show that it is equal to $CC(f_*(\gamma))$. 
		
		Consider the maps of vector bundles on $\widetilde{X}$:
		\[
		\pi_X^*f^*T^*B=\tilde{f}^*\pi_B^*T^*B\to \tilde{f}^*T^*E\to T^*\widetilde{X},
		\]
		induced by $\tilde{f}$ and $\pi_B$. Since $\tilde{f}$ is $\Z^k$-equivariant, there is a natural $\Z^k$-action on $\tilde{f}^*T^*E$. There are also obvious $\Z^k$-actions on $\pi_X^*f^*T^*B$ and $T^*\widetilde{X}$, and moreover, the above two maps are $\Z^k$ equivariant. Taking quotient of the $\Z^k$-action, we have a commutative diagram
		\[
		\xymatrix{
			\tilde{f}^*\pi_B^*T^*B\ar[r]^{\tilde{\iota}}\ar[d]&\tilde{f}^*T^*E\ar[r]^{\tilde{v}_1}\ar[d]^{\pi_E}& T^*\widetilde{X}\ar[d]^{\pi_T}\\
			f^*T^*B\ar[r]^{\iota}&\tilde{f}^*T^*E/\Z^k \ar[r]^{\;\;v_1}&T^*X
		}
		\]
		where the first row consists of vector bundles on $\widetilde{X}$, the second row consists of vector bundles on $X$, all horizontal maps are vector bundle maps and vertical maps are covering maps. 
		
		\begin{claim}\label{claim0}
			The set-theoretic preimage $v_1^{-1}(\Lambda)$ is an analytic subset of dimension at most $\dim E$. 
		\end{claim}
		\begin{proof}[Proof of the claim]
			Let $\widetilde{\Lambda}=\pi_T^{-1}(\Lambda)$. Then $\widetilde{\Lambda}=T^*_{\pi_X^{-1}Z}\widetilde{X}$ is a conic Lagrangian (not necessarily connected) subvariety of $T^*\widetilde{X}$.
			
			Consider the maps induced by $\tilde{f}: \widetilde{X}\to E$:
			\[
			T^*\widetilde{X}\xleftarrow{\tilde{v}_1}\tilde{f}^*T^*E\xrightarrow{\tilde{v}_2} T^*E.
			\]
			It follows from \cite[Proposition 4.9]{Kas0} that there is a locally finite conic Lagrangian cycle $\Lambda'$ in $T^*E$ such that $\tilde{v}_1^{-1}(\widetilde{\Lambda})\subset v_2^{-1}(\Lambda')$. The property (1) of the proposition implies that the preimage of any point under $\tilde{v}_2$ is also a discrete set. Thus, 
			\[
			\dim \tilde{v}_1^{-1}(\widetilde{\Lambda})\leq \dim v_2^{-1}(\Lambda')\leq \dim \Lambda'=\dim E,
			\]
			where $\dim E$ is the dimension of the total space of $E$. 
			Since $\tilde{v}_1^{-1}(\widetilde{\Lambda})$ is a covering space of $v_1^{-1}(\Lambda)$, it follows that $\dim v_1^{-1}(\Lambda)\leq \dim E$. %Here, $\dim E$ is the dimension of the total space of $E$, not the rank of $E$. 
		\end{proof}
		
		Consider the following maps induced by $f: X\to B$:
		\[
		T^*X \xleftarrow{u_1} f^*T^*B\xrightarrow{u_2}T^*B.
		\]
		Notice that the composition $\tilde{v}_1\circ \tilde{\iota}: (\tilde{f}\circ \pi_B)^*T^*B\to T^*\widetilde{X}$ is equal to the natural pullback map of cotangent bundle induced by $\tilde{f}\circ \pi_B$. Thus, the quotient map $v_1\circ \iota: f^*T^*B\to T^*X$ is equal to the natural pullback map $u_1$. 
		
		Notice that on $E$, we have a short exact sequence of holomorphic vector bundles
		\[
		0\to \pi_B^*T^*B\to T^*E\to \pi_B^*E^\vee\to 0,
		\]
		where $E^\vee$ is the dual vector bundle of $E$. 
		Taking the pullback to $\widetilde{X}$, we have
		\[
		0\to \tilde{f}^*\pi_B^*T^*B\to \tilde{f}^*T^*E\to \tilde{f}^*E^\vee\to 0.
		\]
		Taking the quotient by the $\Z^k$-action, we have a short exact sequence of vector bundles on $X$:
		\[
		0\to f^*T^*B\to \tilde{f}^*T^*E/\Z^k\to f^*E^\vee\to 0. 
		\]
		
		Now, take the scheme-theoretic preimage $v_1^{-1}(\Lambda)$, which, by Claim \ref{claim0}, can be regarded as an effective analytic $\dim(E)$-cycle in $\tilde{f}^*T^*E/\Z^k$. By the above argument and Lemma \ref{lemma_ef}, there exists an effective analytic $\dim(B)$-cycle $\Lambda''$ on $\iota^{-1}v_1^{-1}(\Lambda)$, representing the class $\iota^{*}[v_1^{-1}(\Lambda)]$. Since $f$ is a proper map, so is $u_2$, and hence $u_{2\star}(\Lambda'')$ is a well-defined $\dim(B)$-analytic cycle in $T^*B$. 
		\begin{claim}
			The analytic $\dim(B)$-cycle $u_{2\star}(\Lambda'')$ is equal to $CC(f_*(\gamma))$. In particular, $u_{2\star}(\Lambda'')$ does not depend on the choice of $\Lambda''$ we made when applying Lemma \ref{lemma_ef}. 
		\end{claim}
		\begin{proof}[Proof of the claim]
			%The claim follows easily from Subsection \ref{sec_push}, once 
			Let us give the precise Borel-Moore homology groups in which each cycle/class is defined. As in \cref{sec_push}, there exist Whitney stratifications $\mathscr{S}$ and $\mathscr{S}'$ of $X$ and $B$ respectively such that $\gamma$ is constructible with respect to $\mathscr{S}$ and $f$ is a stratified map with respect to $\mathscr{S}$ and $\mathscr{S}'$. By possibly shrinking $B$, we can assume that both $\mathscr{S}$ and $\mathscr{S}'$ are finite stratifications. 
			
			Now, we can regard $CC(\gamma)=T^*_Z X=\Lambda$ as an element in $H^{BM}_{2\dim X}(T^*_{\mathscr{S}}X, \Z)$. Then the scheme-theoretic preimage $v_1^{-1}(\Lambda)$, considered as an analytic $\dim(E)$-cycle, represents the element $v_1^*(\Lambda)\in H^{BM}_{2\dim E}(v_1^{-1}(T^*_{\mathscr{S}}X), \Z)$. Since $v_1 \circ \iota=u_1$, we have
			\[
			\iota^*v_1^*(\Lambda)=u_1^*(\Lambda)=[\Lambda'']\in H^{BM}_{2\dim B}(u_1^{-1}(T^*_{\mathscr{S}}X), \Z).
			\]
			Then, 
			\begin{equation}\label{eq_BM}
				u_{2*}u_1^*(\Lambda)=[u_{2\star}(\Lambda'')]\in H^{BM}_{2\dim B}(T^*_{\mathscr{S}'}B, \Z).
			\end{equation}
			By \cref{sec_push}, $u_{2\star}(\Lambda'')$ and $CC(f_*(\gamma))$ represent the same homology classes in $H^{BM}_{2\dim B}(T^*_{\mathscr{S}'}B, \Z)$. Since $T^*_{\mathscr{S}'}B$ is the union of finitely many $\dim(B)$-dimensional irreducible varieties, $H^{BM}_{2\dim B}(T^*_{\mathscr{S}'}B, \Z)$ is generated by their fundamental classes, and hence the equality \eqref{eq_BM} implies that $u_{2\star}(\Lambda'')=CC(f_*(\gamma))$ as analytic $\dim(B)$-cycles. 
		\end{proof}
		By the construction, $u_{2\star}(\Lambda'')$ is evidently effective. Thus, $f_*(\gamma)$ is CC-effective.  The proposition is proved. 
	\end{proof}
	\begin{proof}[Proof of \cref{prop_linear}] 
		Under the assumptions of Proposition \ref{prop_linear}, we can construct the following commutative diagram analogous to diagram \eqref{eq_diagram},
		\begin{equation}\label{eq_big}
			\xymatrix{
				\widetilde{Z}_{\norm}\ar[d]\ar[r]^{\tilde{p}}&\widetilde{Z}\ar[d]\ar[r]^{\iota}&\widetilde{X}\ar[d]\ar[r]^{\tilde\phi}&\cD_1\times \cM\ar[d]^{\pi_M}\\
				\widetilde{Z}_{\norm}^{ss}\ar[d]\ar[r]^{\tilde{p}^{ss}}&\widetilde{Z}^{ss}\ar[r]^{\iota^{ss}}&\widetilde{X}^{ss}\ar[r]^{\tilde\phi^{ss}}&\cD_1\times \cD_2\ar[d]^{\pi_{\Gamma}}\\
				Z_{\norm}\ar[rrr]^{\phi}&&&\Gamma\backslash\cD_1\times \cD_2
			}
		\end{equation}
		where
		\begin{itemize}
			\item $\Gamma$ is the image of the composition 
			\[
			\pi_1(Z_{\norm})\xrightarrow{\tau} \pi_1(X)\xrightarrow{\sigma_1\times \sigma_2^{ss}} \GL(N_1, \C)\times \GL(N_2, \C),
			\]
			\item $\widetilde{X}$ and $\widetilde{X}^{ss}$ are the covering spaces of $X$ induced by $\ker[\sigma_1\times \sigma_2]$ and $\ker[\sigma_1\times \sigma_2^{ss}]$ respectively;
			
			\item $\widetilde{Z}_{\norm}$ and $\widetilde{Z}_{\norm}^{ss}$ are the covering spaces of $Z_{\norm}$ induced by $\ker[ (\sigma_1\times \sigma_2)\circ \tau]$ and $\ker[(\sigma_1\times \sigma_2^{ss})\circ\tau]$ respectively;
			
			\item choosing a lifting $\widetilde{Z}_{\norm}^{ss}\to \widetilde{X}^{ss}$ of $Z_{\norm}\to X$,  let $\widetilde{Z}_{\norm}^{ss}\xrightarrow{\tilde{p}^{ss}}\widetilde{Z}^{ss}\xrightarrow{\iota^{ss}}\widetilde{X}^{ss}$ be the factorization such that $\tilde{p}^{ss}$ is surjective and $\iota^{ss}$ is injective;
			
			\item choosing a lifting $\widetilde{Z}_{\norm}\to \widetilde{X}$ of $\widetilde{Z}_{\norm}^{ss}\to \widetilde{X}^{ss}$,  let $\widetilde{Z}_{\norm}\xrightarrow{\tilde{p}}\widetilde{Z}\xrightarrow{\iota}\widetilde{X}$ be the factorization such that $\tilde{p}^{ss}$ is surjective and $\iota^{ss}$ is injective;
			%$\widetilde{Z}$ is the image of $\widetilde{Z}_{\norm}$ in $\widetilde{X}$, and $\widetilde{Z}^{ss}$ is the image of $\widetilde{Z}_{\norm}^{ss}$ in $\widetilde{X}^{ss}$,
			%\item $\iota$ and $\iota^{ss}$ are the inclusion maps;
			\item $\tilde\phi$, $\tilde{\phi}^{ss}$ and $\phi$ are the period maps.
			%\item $\tilde{p}$ and $\tilde{p}^{ss}$ are normalization maps. 
		\end{itemize}
		Notice that all vertical maps are covering maps except $\pi_M$.  Moreover, $\phi:Z_{\rm norm}\to \Gamma\backslash\cD_1\times \cD_2$,   $\tilde{\phi}\circ \iota\circ \tilde{p}:\widetilde{Z}_{\rm norm}\to \cD_1\times \cM$ and $\tilde{\phi}^{ss}\circ \iota^{ss}\circ \tilde{p}^{ss}:\widetilde{Z}_{\rm norm}^{ss}\to \cD_1\times \cM$ are all proper holomorphic maps. By definition, $\Gamma$ is a subgroup of $\Gamma_1\times \Gamma_2$. Since both $\Gamma_1$ and $\Gamma_2$ are discrete and torsion-free, so is $\Gamma$.

		\begin{claim}\label{prop_leff}
			Let $\delta$ be the constructible function on $Z_{\norm}$ as defined in Definition \ref{defn_delta}. Under the above notations, $\phi_{*}(\delta)$ is CC-effective. 
		\end{claim}
		\begin{proof}
			Let $\tilde{\delta}$ be the pullback of $\delta$ to $\widetilde{Z}_{\norm}^{ss}$.  %Let $\tilde{\delta}^{ss}=\tilde{p}^{ss}_*(\tilde{\delta}^{ss}_{\norm})$ and $\tilde{\delta}=\tilde{p}_*(\tilde{\delta}_{\norm})$. 
			As an analog of Lemma \ref{lemma_tildep}, $\tilde{p}^{ss}_*(\tilde{\delta})$ is CC-effective. 
			Since $\iota^{ss}$ is a closed embedding, $(\iota^{ss}\circ\tilde{p}^{ss})_*(\tilde{\delta})$ is also CC-effective. 
			
			Let $\widetilde{Z}_{\rm norm}^{\rm univ}\to \cD_1\times \cM$ be  the product of the period map of $p^*\sigma_1$ and the mixed period map of $p^*\sigma_2$.  By the assumptions in Item (i)  it has  discrete fibers. Note that  it factors through the proper map  $\tilde{\phi}\circ \iota\circ \tilde{p}:\widetilde{Z}_{\rm norm}\to \cD_1\times \cM$ is proper and the \'etale cover $\widetilde{Z}_{\rm norm}^{\rm univ}\to  \widetilde{Z}_{\rm norm} $. It implies that $\tilde{\phi}\circ \iota\circ \tilde{p}: \widetilde{Z}_{\norm}\to \cD_1\times \cM$ is a finite morphism.  Since $\tilde{\phi}$ is surjective, $\iota\circ \tilde{p}: \widetilde{Z}\to \cD_1\times \cM$ is also a finite morphism. 
			
			By definition, the deck transfomration group of the normal covering map $\widetilde{Z}_{\norm}^{ss}\to Z_{\norm}$ is equal to $\Gamma$. Thus, the bottom rectangle of \eqref{eq_big} is Cartesian. 
			%			\begin{equation}\label{eq_product}
				%				\widetilde{Z}_{\norm}^{ss}=Z_{\norm}\times_{\Gamma\backslash \cD_1\times \cD_2} \cD_1\times \cD_2.
				%			\end{equation}
			Since $Z_{\norm}$ is projective, $\phi$ is proper, and hence $\tilde{\phi}^{ss}\circ \iota^{ss}\circ \tilde{p}^{ss}: \widetilde{Z}^{ss}_{\norm}\to \cD_1\times \cD_2$ is also proper. Since $\tilde{p}^{ss}$ is surjective, $\tilde{\phi}^{ss}\circ\iota^{ss}: \widetilde{Z}^{ss}\to  \cD_1\times \cD_2$ is also proper. 
			
			%By definition, $\widetilde{X}\to \widetilde{X}^{ss}$ is a normal covering map with deck transformation group being $\ker[\sigma_1\times \sigma_2^{ss}]/\ker[\sigma_1\times \sigma_2]$, which is an abelian group. 
			
			Denote the kernel of the natural map $\im(\sigma_1\times \sigma_2)\to \im (\sigma_1\times \sigma_2^{ss})$ by $\Gamma_0$. Then the monodromy action of $\Gamma_0$ on $\cD_1\times \cD_2$ is trivial. By definition, $\Gamma_0$ is also equal to the deck transformation group of the normal covering map $\widetilde{X}\to \widetilde{X}^{ss}$.  Since the period maps are compatible with the monodromy actions, the map $\tilde{\phi}$ is $\Gamma_0$-equivariant. As well-known facts for mixed period domain and mixed period maps (see e.g. \cite[Section 2]{Her99}),  the map $\cM\to \cD_2$, and hence $\pi_M: \cD_1\times \cM\to \cD_1\times \cD_2$, are vector bundles. Moreover, $\Gamma_0$ acts on the fibers of $\pi_M$ by linear translations. 
			
			%By the above arguments together with \Cref{claim:translate},
			Therefore, the constructible function $(\iota^{ss}\circ\tilde{p}^{ss})_*(\tilde{\delta})$ satisfies the axiom of Proposition \ref{prop_E} with respect to the most upper-right square of \eqref{eq_big} restricted to a small ball in $\cD_1\times \cD_2$. Since  CC-effectiveness is a local property, Proposition \ref{prop_E} implies that 
			\[
			\tilde{\phi}^{ss}_*(\iota^{ss}\circ\tilde{p}^{ss})_*(\tilde{\delta})=\tilde{\phi}^{ss}_*\iota^{ss}_{*}\tilde{p}^{ss}_*(\tilde{\delta})
			\]
			is CC-effective. 
			%			Since $\widetilde{Z}^{ss}_{\norm}$ is a connected component of the fiber product \eqref{eq_product}, near the image of $\tilde{\phi}^{ss}\circ \iota^{ss}\circ \tilde{p}^{ss}$, 
			Since the bottom rectangle of \eqref{eq_big} is Cartesian,
			$\pi_{\Gamma}^*(\phi_{*}(\delta))$ is equal to $\tilde{\phi}^{ss}_*\iota^{ss}_{*}\tilde{p}^{ss}_*(\tilde{\delta})$. Thus, $\pi_{\Gamma}^*(\phi_{*}(\delta))$ is CC-effective.
			%near the image of $\tilde{\phi}^{ss}\circ \iota^{ss}\circ \tilde{p}^{ss}$. 
			%Since $\pi_{\Gamma}^*(\phi_{*}(\delta))$ is invariant under the deck transformation of the normal covering map $\pi_{\Gamma}$, $\pi_{\Gamma}^*(\phi_{*}(\delta))$ is CC-effective. 
			Since $\pi_{\Gamma}$ is a covering map, we can conclude that $\phi_{*}(\delta)$ is CC-effective. 
		\end{proof}

		We can follow the same arguments as in the proof of Proposition \ref{prop_geq0}. By  \Cref{prop_leff} and \cref{prop_AW}, we know that 
		$\chi(\phi_{*}(\delta))\geq 0$. The equations \eqref{eq_01} and \eqref{eq_02} also apply here,  and hence we have 
		\[
		(-1)^{\dim Z}\chi(Eu_Z)=\chi(\delta)=\chi(\phi_{*}(\delta))\geq 0.\qedhere
		\]
	\end{proof}
	
	\subsection{Techniques from the linear Shafarevich conjecture}
	In this subsection, we show the mixed analogue of \cref{prop:Shafarevich},
	using the techniques in the proof of the linear Shafarevich conjecture in \cite{EKPR12}. 
	
	\begin{proposition} \label{prop:linear}
		Let $X$ be a projective manifold and let $\rho:\pi_1(X)\to \GL(r,\C)$ be a large  representation.   Then there exists 
		\begin{itemize}
			\item  a  semisimple representation $\sigma_1:\pi_1(X)\to \GL(N_1,\bC)$  underlying a $\bC$-VHS;
			\item  a representation $\sigma_2:\pi_1(X)\to \GL(N_2,\bC)$   such that the associated local system $L_{\sigma_2}$ underlies an $\bR$-VMHS of weight $-1,0$;
			\item and a multivalued {closed} holomorphic one-form $\eta$ on $X$ 
		\end{itemize}such that for any irreducible subvariety $Z$ of $X$, when $\eta|_{Z_{\rm reg}}$ is trivial,  
		\begin{thmlist}   
			\item  \label{item:1}
			let $p:Z_{\rm norm}\to Z$ be the normalization. Then  $p^*\sigma_1:\pi_1(Z_{\rm norm})\to \GL(N_1,\bC)$ and $p^*\sigma_2^{ss}:\pi_1(Z_{\rm norm})\to \GL(N_2,\bC)$ both have discrete images.  %Here $\sigma_2^{ss}$ is the semisimplification of $\sigma_2$. 
			%	 	\item \label{item:2} let $\Gamma_1:=\sigma_1({\rm Im}[\pi_1(Z_{\rm norm}\to \pi_1(X)])$, $\Gamma_2:=\sigma_2({\rm Im}[\pi_1(Z_{\rm norm})\to \pi_1(X)])$ and $\Gamma_2^{ss}:=\sigma^{ss}_2({\rm Im}[\pi_1(Z_{\rm norm})\to \pi_1(X)])$. Then   the natural   homomorphism $\Gamma_2\twoheadrightarrow \Gamma^{ss}_2$ is surjective,  whose kernel  $N$ is an abelian group.  
			\item \label{item:3} Letting $\phi_1:\widetilde{Z}^{\rm univ}_{\rm norm}\to   \cD_1$ be the period map of $L_{\sigma_1}$ and $\phi_2:\widetilde{Z}^{\rm univ}_{\rm norm}\to   \cM$ be the  mixed period map of $L_{\sigma_2}$, then 
			\begin{align}\label{eq:s3}
				( \phi_1,\phi_2):\widetilde{Z}^{\rm univ}_{\rm norm}\to    \cD_1\times   \cM
			\end{align}  has discrete fibers. %Here $\widetilde{Z}^{\rm univ}_{\rm norm}$ is the universal cover of $Z_{\norm}$.
			%		\item Let $\cD_2:=\cD({\rm Gr}_0^W\cV)\times \cD({\rm Gr}_1^W\cV)$ be  the period domain corresponding to the graded parts of $\sigma_2$.  Then $  \cM\to   \cD_2$ is an affine bundle with fibers denoted by $V$. 
		\end{thmlist}	 
	\end{proposition} 
	\begin{proof}
		Let $s_{{\rm fac},r}:X\to S_{{\rm fac},r}$ be the factorization map defined in Definition \ref{def:reduction ac}. By the same arguments in the proof of Proposition \ref{prop:Shafarevich}, there exists a closed multivalued one-form $\eta$ on $X$ such that for any closed subvariety $Z$, $\eta|_{Z_{\rm reg}}$  is trivial if and only if $s_{{\rm fac},r}(Z)$ is a point. 
		From now on, we will assume that $s_{{\rm fac},r}(Z)$ is a point.  
		
		In \cite[Proposition 3.13]{DY23}, the authors constructed 
		\begin{itemize}
			\item  a  semisimple representation $\sigma_1:\pi_1(X)\to \GL(N_1,\bC)$   underlying a $\bC$-VHS;
			\item  semisimple representations $\{\rho_i:\pi_1(X)\to \GL(r,k)\}_{i=1,\ldots,n}$ where $k$ is a number field,
		\end{itemize}
		with the following properties.
		\begin{enumerate}
			\item  \label{item:conjugate2}The image of $p^*\rho_i:\pi_1(Z_{\norm})\to \GL(r,k)$ is contained in $\GL(r,\cO_k)$, where $\cO_k$ is the ring of integer of $k$. Moreover,   the direct sum
			$$
			\bigoplus_{i=1}^{k} \bigoplus_{w\in {\rm Ar}(k)}p^*\rho_{i,w}:\pi_1(Z_{\rm norm})\to \prod\GL(r,\bC)			
			$$ 
			is conjugate to $p^*\sigma_1$, where ${\rm Ar}(k)$ is the set of archimedean places of $k$ and $\rho_{i,w}:=w\rho_i$ for each $w\in {\rm Ar}(k)$.  In particular, $p^*\sigma_1:\pi_1(Z_{\rm norm})\to \GL(N_1,\bC)$ is a $\bC$-VHS with discrete monodromy.
			\item  \label{item:conjugate3} Each geometric connected component of $M$ contains some $[\varrho_{i,w}]$. 
			\item \label{item:conjugate} For any semisimple representation $\tau:\pi_1(X)\to \GL(r,\bC)$ such that $[\tau]$ is in the same geometric component of $[\varrho_{i,w}]$, $p^*\tau$ is conjugate to $p^*\rho_{i,w}$.
		\end{enumerate}
		The representation $\sigma_2:\pi_1(X)\to \GL(N_2,\bR)$ underlying $\bR$-VMHS   of weight $-1,0$ is constructed in \cite[Lemma 5.4]{EKPR12}, and let us briefly recall it. We define $M^{\rm VHS}$ to be the set of semisimple representations $\pi_1(X)\to \GL(r,\bC)$ underlying a $\bC$-VHS. Define $\mathcal{T}_M^{\rm VHS}$ to be the set of    the tensor product $V_1\otimes\cdots\otimes V_n$ with each $V_i$ in $M^{\rm VHS}$.  In \cite[Lemma 5.4]{EKPR12},  $\sigma_2$ is defined to be the monodromy representation of  $\sum_{i=1}^{\ell} (\mathbb{D}_1(\bV_{i})+\overline{\mathbb{D}_1(\bV_i)}))$ where $\{\bV_i\}_{i=1,\ldots,\ell}$ are certain objects in $\mathcal{T}_M^{\rm VHS}$. Here  $\mathbb{D}_1(\bV_i)$ is the 1-step $\bC$-MVHS (hence of weight length 1) constructed in \cite[Definition 2.11]{EKPR12} whose graded part is a direct sum of $\bV_i$ (cf. also \cite[Theorem 3.15]{ES11}),    and $\overline{\mathbb{D}_1(\bV_i)}$ a $\bC$-VMHS which is the conjugate of $\mathbb{D}_1(\bV_i)$. Hence $L_{\sigma_2}$ is an $\bR$-VMHS of weight length 1.
		
		By the definition of $\mathcal{T}^{\rm VHS}$, we know that there exists $\{\varrho_n:\pi_1(X)\to \GL(r,\bC)\}_{n=1,\ldots,m}$ underlying $\bC$-VHS such that 
		$
		V_j=L_{\varrho_1}\otimes\cdots\otimes L_{\varrho_m}.
		$  By \Cref{item:conjugate3}, there exists $\varrho_{j_n,w_n}$ such that $[\varrho_{j_n,w_n}]$ is in the same geometric component of $[\varrho_n]$. 
		For any $w\in {\rm Ar}(k)$, by Simpson's ubiquity \cite[Theorem 3]{Simpson}, we know that there exists a semisimple representation $\varrho_{j_n,w}^{\rm vhs}:\pi_1(X)\to \GL(r,\bC)$ underlying a $\bC$-VHS such that  $[\varrho_{j_n,w}^{\rm vhs}]$ is in the same geometric connected component of  $[\varrho_{j_n,w}]$ for each $w\in {\rm Ar}(k)$.  Let $L_n$ be the local system associated to the direct sum representation
		$$
		\bigoplus_{w\in {\rm Ar}(k),w\neq w_n}\varrho_{j_n,w}^{\rm vhs}\oplus\varrho_n.
		$$ 
		Then $L_n$ is a $\bC$-VHS. Moreover,  by \Cref{item:conjugate},
		$p^*L_n$ is isomorphic to the local system corresponding to   the representation
		$$
		\bigoplus_{w\in {\rm Ar}(k)}	p^*\varrho_{j_n,w}:\pi_1(Z_{\rm norm})\to \prod\GL(r,\bC).
		$$
		By \Cref{item:conjugate2}, we know that $p^*L_n$ has discrete monodromy.  Note that  $\bV_j$ is a direct factor of  $L_1\otimes\cdots\otimes L_m$.   We can replace each $\bV_j$  by  $L_1\otimes\cdots\otimes L_m$ and \cite[Lemma 5.4]{EKPR12} will still holds.   In this case,    $p^*{\sigma^{ss}_2}$ has discrete image by our construction.   
		\Cref{item:1} is proved. 
		
		Let $M:=M_{\rm B}(\pi_1(X),\GL_r)(\bC)$. 	We define $ \widetilde{H_M^0}$   to be the intersection of the kernels of all semisimple representations $\pi_1(X)\to \GL(r,\bC)$. Let  $\widetilde{H_M^1}\subset \widetilde{H_M^0}$ be the intersection of   $ \widetilde{H_M^0}$ and the kernels of the monodromy representation of $\mathbb{D}_1(L_\sigma)$ with $\sigma\in \mathcal{T}_M^{\rm VHS}$. Let $\pi_X:\widetilde{X}^{\rm univ} \to X$ be the universal covering map. Denote by $\widetilde{X_M^i}:=\widetilde{X}^{\rm univ}/\widetilde{H_M^i}$ for $i=0,1$ and let $\pi_i:\widetilde{X_M^i}\to X$ be the covering map. 
		In \cite[Lemma 3.30]{DY23}, it is proved that 
		\begin{claim}
			Each connected component of the fiber of the holomorphic map \begin{align}\label{eq:s1}
				(s_{{\rm fac},r}\circ \pi_0,\tilde{\varphi}_1):\widetilde{X_M^0}\to S_{{\rm fac},r}\times \cD_1 
			\end{align}   is compact.  Here $\tilde{\varphi}_1:\widetilde{X_M^0}\to \cD_1$ is the period map of $L_{\sigma_1}$. 
		\end{claim}		 
		By the generalized Stein factorization discussed in \Cref{sec_Stein}, the set $\widetilde{S}_M$  of connected components of fibres of   can be endowed with the structure of a normal complex space such that  \eqref{eq:s1} factors through a proper holomorphic fibration $\Psi_0:\widetilde{X_M^0}\to \widetilde{S^0_M}$. The Galois group $\pi_1(X)/\widetilde{H_M^0}$  induces a properly discontinuous action on $\widetilde{S^0_M}$ such that $\Psi_0$ is equivariant. By \cite[Proposition 3.13]{DY23}, the \emph{reductive Shafarevich morphism}  ${\rm Sh}_M:X\to {\rm Sh}_M(X)$  of $M$ is the quotient of $\Psi_0$ by this action $\pi_1(X)/\widetilde{H_M^0}$. Namely, for any subvariety $Z$ of $Z$, ${\rm Sh}_M(Z)$ is a point if and only if $\tau({\rm Im}[\pi_1(Z_{\rm norm})\to \pi_1(X)])$ is finite for any  semisimple representation $\tau:\pi_1(X)\to \GL(r,\bC)$. 
		
		Note that the mixed period map $\varphi_2:\widetilde{X}^{\rm univ}\to \cM$ of $L_{\sigma_2}$ factors through $\tilde{\varphi}_2:\widetilde{X_M^1}\to \cM$.  
		It is proved in \cite[Lemma 5.5]{EKPR12} that each connected component of the fiber of the holomorphic map \begin{align}\label{eq:s2}
			(\tilde{\varphi}_2,{\rm sh}_M\circ \pi_1):\widetilde{X_M^1}\to\cM\times {\rm Sh}_M(X) \end{align}  
		is compact.    By the generalized Stein factorization again, the set $\widetilde{S ^1_M}$  of connected components of fibres of   can be endowed with the structure of a normal complex space such that  \eqref{eq:s2} factors through a proper holomorphic fibration $\Psi_1:\widetilde{X_M^1}\to \widetilde{S^1_M}$. The Galois group $\pi_1(X)/\widetilde{H_M^1}$  induces a properly discontinuous action on $\widetilde{S_M^1}$ such that $\Psi_1$ is equivariant.  Let  ${\rm sh}^1_M:X\to {\rm Sh}^1_M(X)$  be the quotient of $\Psi_1$ by this action $\pi_1(X)/\widetilde{H_M^1}$.  Let $H$ be the intersection of the kernels of all linear representations $\rho:\pi_1(X)\to \GL(r,\bC)$.   A crucial fact proven  in \cite[Proposition 3.10 \& p.1549]{EKPR12} is that,  ${\rm sh}^1_M:X\to {\rm Sh}^1_M(X)$ is the Shafarevich morphism ${\rm sh}_H:X\to {\rm Sh}_H(X)$ of $(X,H)$.  
		
		\begin{claim}
			${\rm sh}_H:X\to {\rm Sh}_H(X)$ is the identity map.
		\end{claim}
		\begin{proof}
			Note that $H\subset \ker\rho$. Hence there is an injection 
			$$
			\pi_1(X)/\ker\rho\to \pi_1(X)/H.
			$$
			Since $\rho:\pi_1(X)\to \GL(r,\bC)$ is  a large representation, 	for any closed subvariety $Z\subset X$,  
			the image $\pi_1(Z_{\rm norm}) \to \pi_1(X)/\ker\rho$ 
			is an infinite group. Hence $ 	{\rm Im}[\pi_1(Z_{\rm norm})\to \pi_1(X)/H]$ is also infinite.  Therefore, the fibers of ${\rm sh}_H:X\to {\rm Sh}_H(X)$  are zero dimensional. Since ${\rm sh}_H$ has connected fibers, and ${\rm Sh}_H(X)$ is normal, it follows that  ${\rm sh}_H$  is the identity map. The claim is proved. 
		\end{proof}
		Therefore, ${\rm sh}^1_M:X\to {\rm Sh}^1_M(X)$ is the identity map. By our construction of ${\rm sh}_M^1$ and ${\rm sh}^0_M$, we conclude that  
		\begin{align} 
			(s_{{\rm fac},r}\circ\pi_X, {\varphi}_1,{\varphi}_2):\widetilde{X}^{\rm univ}\to S_{{\rm fac},r}\times  \cD_1\times   \cM
		\end{align}  has discrete fibers, where $\varphi_1:\widetilde{X}^{\rm univ}\to \cM$ is the period map of $L_{\sigma_1}$. 
		Since $s_{{\rm fac},r}(Z)$ is a point,  it follows that  
		\begin{align*} 
			( \phi_1,\phi_2):\widetilde{Z}^{\rm univ}_{\rm norm}\to    \cD_1\times   \cM
		\end{align*}  has discrete fibers.  \Cref{item:3} is proved.  We complete the proof of the proposition. 
	\end{proof}

	\subsection{Proof of Theorem \ref{thm_main} in the case when $K=\C$ and $\rho$ is linear}
	After proving Propositions \ref{prop_linear} and \ref{prop:linear}, we are in the same senario as in the semisimple case. So we only give a sketch. 
	\begin{proof}[Proof of Theorem \ref{thm_main} in the case when $K=\C$ and $\rho$ is not necessarily semisimple] We  apply \cref{prop:linear} to construct a   semisimple representation $\sigma_1:\pi_1(X)\to \GL(N_1,\bC)$   underlying a $\bC$-VHS, 
		a    representation $\sigma_2:\pi_1(X)\to \GL(N_2,\bR)$  underlying a $\bR$-VMHS of weight $-1,0$, and  a multivalued   one-form $\eta$ on $X$ such that they satisfies the properties therein.   We  use the vanishing cycle functor $\Phi_\eta$ and apply \cref{cor_subvariety} repeatedly so that we reduce the proof to the case where $\eta|_{Z}$ is trivial. Then the properties in \cref{prop:linear} are fulfilled.  We then apply \cref{prop_linear} to conclude that $T_Z^*X \cdot T_X^* X\geq 0$.
	\end{proof}

%	\bibliography{biblio}
%	\bibliographystyle{amsalpha}

	\providecommand{\bysame}{\leavevmode\hbox to3em{\hrulefill}\thinspace}
	\providecommand{\MR}{\relax\ifhmode\unskip\space\fi MR }
	\providecommand{\MRhref}[2]{%
		\href{http://www.ams.org/mathscinet-getitem?mr=#1}{#2}
	}
	\providecommand{\href}[2]{#2}

\end{document}